\documentclass[12pt]{amsart}
\usepackage{graphicx}
\usepackage{latexsym}
\usepackage{amssymb}
\usepackage{gensymb}
\usepackage{amsmath,amsthm,yhmath}
\usepackage{verbatim}
\usepackage{pb-diagram,amsfonts,hyperref}
\usepackage{tikz}
\usepackage[active]{srcltx}

\usepackage{cite}
\RequirePackage[normalem]{ulem} 
\numberwithin{equation}{section}
\newtheorem{theorem}{Theorem}[section]
\newtheorem{Lemma}[theorem]{Lemma}
\newtheorem{corollary}[theorem]{Corollary}

\theoremstyle{definition}
\newtheorem{definition}[theorem]{Definition}
\theoremstyle{remark}
\newtheorem{remark}[theorem]{Remark}

\newcommand{\la}{\lambda}
\newcommand{\R}{\mathbb{R}}
\newcommand{\Z}{\mathbb{Z}}
\newcommand{\uc}{\mathbb{S}}

\newcommand{\D}{\mathbb{D}}
\newcommand{\N}{\mathbb{N}}
\newcommand{\M}{\mathcal{M}}
\newcommand{\A}{\mathcal{A}}
\newcommand{\C}{\mathbb{C}}

\newcommand{\lam}{\mathcal{L}}
\newcommand{\hlam}{\widehat{\mathcal{L}}}
\newcommand{\qml}{\mathrm{QML}}
\newcommand{\cg}{\mathrm{CG}}
\newcommand{\sh}{\mathrm{SH}}
\newcommand{\sm}{\setminus}
\newcommand{\ch}{\mathcal{H}}
\newcommand{\Bd}{\partial}
\newcommand{\ol}{\overline}
\newcommand{\disk}{\mathbb{D}}
\newcommand{\cdisk}{\overline{\mathbb{D}}}
\newcommand{\V}{\mathcal{V}}
\newcommand{\si}{\sigma}
\newcommand{\Ga}{\Gamma}
\newcommand{\ga}{\gamma}
\newcommand{\hell}{\hat{\ell}}
\newcommand{\tell}{\Tilde{\ell}}
\newcommand{\hc}{\hat{c}}

\newcommand{\hy}{\hat{y}}
\newcommand{\ty}{\tilde{y}}
{\par\noindent\textit{Proof of (#1)}%
}

\begin{document}
\title{Symmetric Cubic Laminations}

\author[\tiny{A.~Blokh}]{Alexander Blokh}

\author[L.~Oversteegen]{Lex Oversteegen}

\author[N.~Selinger]{Nikita Selinger}

\author[V.~Timorin]{Vladlen Timorin}

\author[S.~Vejandla]{Sandeep Chowdary Vejandla}


\address[Alexander~Blokh, Lex~Oversteegen, Nikita Selinger, and Sandeep Chowdary Vejandla]
{Department of Mathematics\\ University of Alabama at Birmingham\\
Birmingham, AL 35294}

\address[Vladlen~Timorin]
{Faculty of Mathematics\\
National Research University Higher School of Economics\\
6 Usacheva str., Moscow, Russia, 119048}

\email[Alexander~Blokh]{ablokh@math.uab.edu}
\email[Lex~Oversteegen]{overstee@uab.edu}
\email[Nikita~Selinger]{selinger@uab.edu}
\email[Vladlen~Timorin]{vtimorin@hse.ru}
\email[Sandeep-Vejandla]{vsc4u@uab.edu}

\subjclass[2010]{Primary 37F20; Secondary 37F10, 37F50}

\keywords{Complex dynamics; laminations; Mandelbrot set; Julia set}




\date{\today}

\begin{abstract}
To investigate the degree $d$ connectedness locus, Thur\-ston studied
\emph{$\sigma_d$-invariant laminations}, where $\sigma_d$ is the
$d$-tupling map on the unit circle, and built a topological model for the
space of quadratic polynomials $f(z) = z^2 +c$. In the spirit of Thurston's
work, we consider the space of all \emph{cubic symmetric polynomials}
$f_\la(z)=z^3+\la^2 z$ in a series of three articles. In the present paper,
the first in the series, we construct a lamination $C_sCL$ together with
the induced factor space $\uc/C_sCL$ of the unit circle $\uc$. As will be
verified in the third paper of the series, $\uc/C_sCL$ is a monotone model
of the \emph{cubic symmetric connected locus}, i.e. the space of all cubic
symmetric polynomials with connected Julia sets.
\end{abstract}

\thanks{The results of this paper are based on the PhD thesis
of Sandeep Vejandla \cite{Vej21}.}

\maketitle
\section{Introduction} A fundamental problem in complex dynamics is to understand
the 
space of complex polynomials of degree $d>1$ modulo affine conjugacy. The
\emph{connectedness locus} $\mathcal M_d$, i.e., the set of all such
polynomials with connected Julia sets, has been extensively studied for the
last 40 years. Major progress has been made for $d=2$ but much less is known
for $d>2$.  Thurston \cite{thu85} introduced \emph{laminations} as a way to
provide models for connected Julia sets and a model for $\M_2$. A
\emph{lamination} $\lam$ is a compact set of chords, called \emph{leaves}, of
the unit circle $\uc$ in the complex plane $\mathbb C$ with the property that
no two leaves intersect inside the open unit disk $\disk$.

Given $d\ge 2$, a lamination is \emph{invariant} if it is preserved by the
map $\sigma_d(z)=z^d$ on the unit circle $\uc$ (see Definition
\ref{inv-lam}). Thurston constructed the space $\qml$ of all \emph{invariant}
quadratic laminations and showed that $\qml$ can be viewed as a lamination
such that for the quotient space $\uc/\qml=\M_2^{Comb}$ there exists a
continuous surjective map $\pi: \partial \mathcal M_2\to \mathcal
M_2^{Comb}$. This map is \emph{monotone}, i.e., all point preimages are
connected, see Definition~\ref{d:monotone}) and is conjecturally a
homeomorphism. Thus, $\mathcal M_2^{Comb}$ is a model of $\M_2$. No such
models are known in case $d>2$.

In this paper we aim at increasing our understanding of $\mathcal M_3$ by
studying a particular slice thereof, namely the slice $\mathcal M_{3,s}$
consisting of all \emph{symmetric} cubic polynomials, i.e., polynomials $P$
with $P(-z)=-P(z)$. These polynomials can be written in the form
$P(z)=z^3+\lambda z$ and correspond to laminations invariant under
$180\degree$ rotation about the origin. Following Thurston, we provide a
model $C_sCL$ for the space of such symmetric cubic invariant laminations and
show that this model is also a lamination (see Figure 1).
Here $C_sCL$ stands for \emph{Cubic$_{symmetric}$ Comajor Lamination}.
\begin{figure}
   \centering
    \begin{minipage}{0.5\textwidth}
        \centering
        \includegraphics[width=1.2\linewidth, height=0.3\textheight]{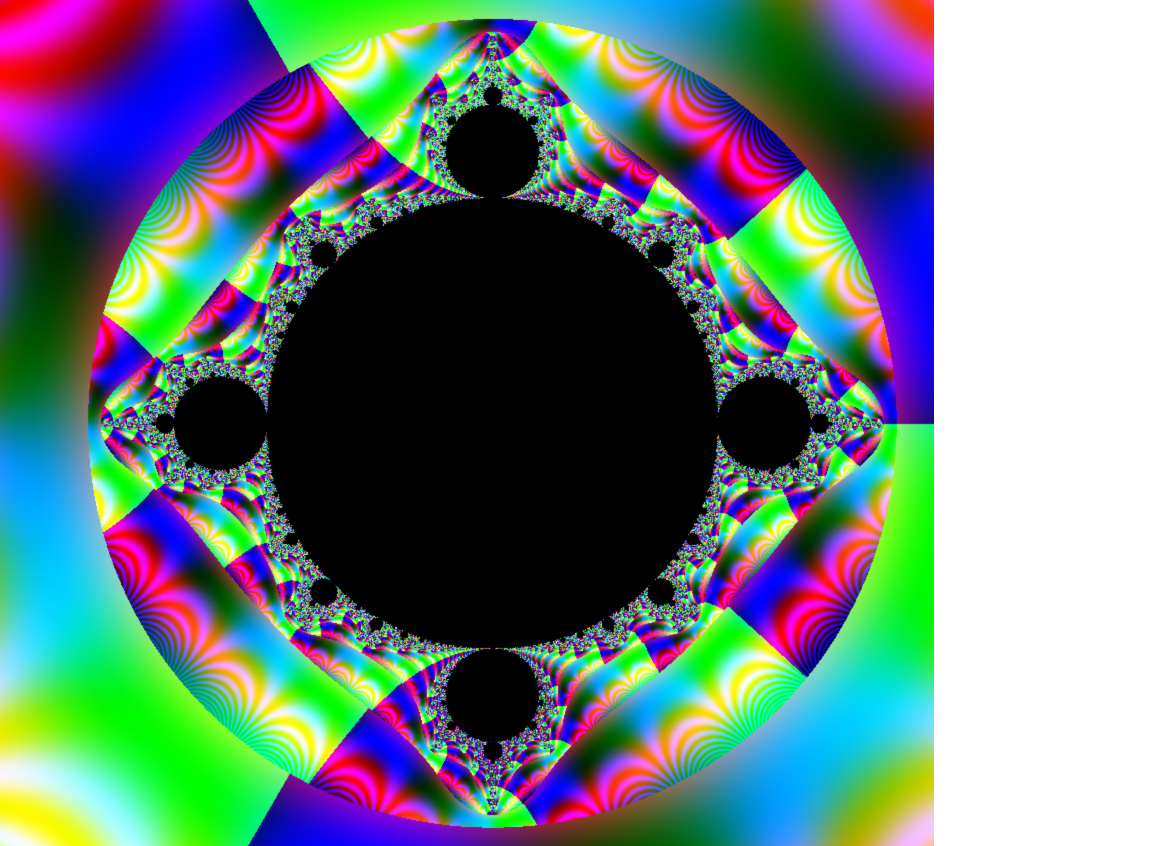}

    \end{minipage}%
    \begin{minipage}{0.5\textwidth}
        \centering
        \includegraphics[width=\linewidth, height=0.3\textheight]{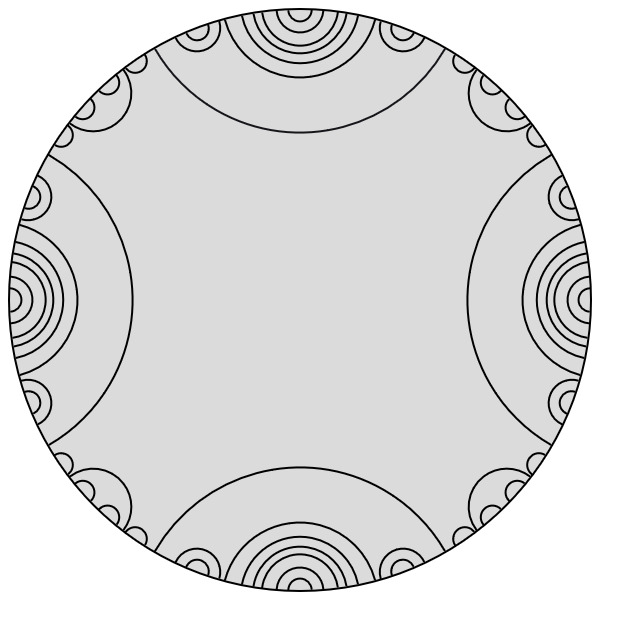}
    \end{minipage}
     \caption{The parameter space of symmetric cubic  polynomials $\mathcal M_{3,s}$ on the left and
     the Cubic symmetric Comajor Lamination $C_sCL$ on the right.}
\end{figure}

Even though the results we obtain are similar to those used in the quadratic
case, there are a lot of interesting distinctions. For example, minors (see
Section 5) of different laminations  may cross in $\disk$, and the first
return maps on finite periodic gaps do not have to be transitive.  We show in
a subsequent paper \cite{botsv3} that there exists a monotone map
$\pi:\partial M_{3,s}\to \mathcal M_{3,s}^{Comb}$ from the boundary of
$\mathcal M_{3,s}$ to the quotient space $\uc/C_sCL$. We also develop in
\cite{botsv2} an algorithm allowing one to explicitly construct $C_sCL$; the
algorithm is related to the famous Lavaurs algorithm \cite{lav89}.

\section{Laminations: classical definitions}

\subsection{Laminational equivalence relation}
\label{ss:ler} Let $\C$ be the complex plane and $\hat{\C}$ be the Riemann
sphere. Let $\D\subset \C$ be the open unit disk and  $P$ be a complex
polynomial of degree $d\geq2$.

\begin{definition}[The Julia set]
The \emph{filled Julia set} $K(P)$ of a polynomial $P$ is the set of all
points $z$ whose orbits do not diverge to infinity under iterations of $P$.
The \emph{Julia set} of $P$ is $J(P)=\partial K(P)$.
\end{definition}{}

\begin{remark}

\begin{enumerate}
    \item We have $P^{-1}(J(P))=P(J(P))=J(P)$.
    \item  The Julia set is the closure of the set of \emph{repelling
        periodic points}.
\end{enumerate}{}

\end{remark}{}

Suppose that the Julia set $J(P)$ is connected. If $f:X\to X$ and $g:Y\to Y$
are self-mappings of topological spaces and there is a continuous surjection
$h:X\to Y$ with $h\circ f=g\circ h$ then $f$ is said to be
\emph{semi-conjugate} to $g$ and the sets $h^{-1}(y)$, where $y\in Y$ are
said to be \emph{fibers} of $h$. If $h$ is a homeomorphism, $f$ is said to be
\emph{conjugate} to $g$. Suppose that $P$ is \emph{monic}, i.e., the leading
term $z^d$ comes with coefficient $1$. By the B\"ottcher theorem, there
exists a conformal map $\Psi:\hat{\C}\sm\overline{\D} \to \hat{\C}\sm K(P)$
that conjugates $\theta_d(z)=z^d$ on $\hat{\C}\sm\overline{\D}$ and
$P|_{\hat{\C}\sm K(P)}$, i.e. $P\circ \Psi=\Psi\circ \theta_d$;  we choose
$\Psi$ so that $\Psi'(\infty)>0$.
\[
\begin{diagram}
\node{\hat{\C}\sm\overline{\D}} \arrow{e,t}{\theta_d} \arrow{s,l}{\Psi} \node{\hat{\C}\sm\overline{\D}}
\arrow{s,l}{\Psi}
\\
\node{\hat{\C}\sm K(P)} \arrow{e,t}{P} \node{\hat{\C}\sm K(P)}
\end{diagram}
\]

From now on (through the end of Section \ref{ss:ler}), let us assume that the
Julia set $J(P)$ is connected and locally connected. Then $\Psi$ extends
continuously to the boundary of the unit disk. Denote this extension by
$\ol{\Psi}$. Let us identify the unit circle $\uc$ with $\R/\Z$. With this
identification, $\sigma_d(t)=dt$ mod 1. 
Define an equivalence relation $\sim_P$ on $\uc$ by setting $x \sim_P y$ if
and only if $\ol{\Psi}(x)=\ol{\Psi}(y)$.

Since $\Psi$ conjugates $\theta_d$ and $P$, the map $\ol{\Psi}$
semi-conjugates $\sigma_d$ and $P|_{J(P)}$, which implies that $\sim_P$ is
$\sigma_d$-invariant. Equivalence classes of $\sim_P$ have pairwise disjoint
convex hulls. The \emph{topological Julia set} $\uc/\sim_P=J(\sim_P)$ is
homeomorphic to $J(P)$, and the \emph{topological polynomial}
$f_{\sim_P}:J(\sim_P)\to J(\sim_P)$, induced by $\sigma_d$, is topologically
conjugate to $P|_{J(P)}$.

\[
\begin{diagram}
\node{\uc/\sim_P} \arrow{e,t}{f_{\sim_P}} \arrow{s,l}{\ol{\Psi}} \node{\uc/\sim_P}
\arrow{s,l}{\ol{\Psi}}
\\
\node{J(P)} \arrow{e,t}{P} \node{J(P)}
\end{diagram}
\]


An equivalence relation $\sim$ on the unit circle, with similar properties to
those of $\sim_P$ above,  can be introduced abstractly without any reference
to the Julia set of a complex polynomial.

\begin{definition}[Laminational equivalence relation]

An equivalence relation $\sim$ on the unit circle $\uc$ is called a
\emph{laminational equivalence relation} if it satisfies the following properties:

\noindent (E1) the graph of $\sim$ is a closed subset in $\uc \times
\uc$;

\noindent (E2) convex hulls in $\D$ of distinct equivalence classes are
disjoint;

\noindent (E3) each equivalence class of $\sim$ is finite.
\end{definition}

A class of equivalence of $\sim$ is called a \emph{$\sim$-class}. For a set
$A\subset \uc$ let $\ch(A)$ be its convex hull. A \emph{chord} $\ol{ab}$ is a
segment connecting points $a, b\in \uc$. An \emph{edge} of $\ch(A)$ is a
chord contained in the boundary of $\ch(A)$. An \emph{edge} of a $\sim$-class
$\mathbf{g}$ is an edge of $\ch(\mathbf{g})$. Given points $a, b\in\uc,$
denote by $(a,b)$ the positively oriented open arc in $\uc$ from $a$ to $b$.

\begin{definition}[Invariance]
\label{d:si-inv-lam}
A laminational equivalence relation $\sim$ is ($\sigma_d$ -){\em in\-va\-riant} if:

\noindent (I1) $\sim$ is {\em forward invariant}: for a $\sim$-class
$\mathbf{g}$, the set $\sigma_d(\mathbf{g})$ is a $\sim$-class;


\noindent (I2) $\sim$ is \emph{backward invariant}: for a $\sim$-class
$\mathbf{g}$, its preimage $\sigma_d^{-1}(\mathbf{g})=\{x\in \uc:
\sigma_d(x)\in \mathbf{g}\}$ is a union of $\sim$-classes;

\noindent (I3) for any $\sim$-class $\mathbf{g}$ with more than two points,
the map $\sigma_d|_{\mathbf{g}}: \mathbf{g}\to \sigma_d(\mathbf{g})$ is a
\emph{covering map with positive orientation}, i.e., for every connected
component $(s, t)$ of $\uc\sm \mathbf{g}$ the arc $(\sigma_d(s),
\sigma_d(t))$ of  the unit circle  is a connected component of $\uc\sm
\sigma_d(\mathbf{g})$.
\end{definition}

\subsection{Invariant laminations}
In Section \ref{ss:ler}, we defined laminational equivalence relations 
based on the identifications of a polynomial map on its locally connected,
and therefore connected, Julia set. A geometric counterpart is the concept of
a \emph{lamination}.



\begin{definition}\label{lam}
A \emph{lamination} $\lam$ is a set of chords in the closed unit disk
$\overline{\D}$, called \emph{leaves} of $\lam$, which satisfies the
following conditions:

\noindent (L1) leaves of $\lam$ 
do not cross; \, (L2) the set $\lam^{\ast}=\cup_{\ell\in\lam}\ell$ is closed.

If (L2) is not assumed then $\lam$ is called a \emph{prelamination}.
\end{definition}

For brevity, in what follows various definitions are given only for
laminations with the understanding that they can be given, verbatim, for
prelaminations as well.

We say that two distinct chords \emph{cross each other} if they intersect inside
the open disk $\D$; such chords are also said to be \emph{linked}. 
 A
\emph{degenerate chord} is a point on $\uc$. Given a chord $\ell
=\overline{ab} \in \lam$, let $\sigma_d(\ell)$ be the chord with endpoints
$\sigma_d(a)$ and $\sigma_d(b)$. If $\sigma_d(a) = \sigma_d(b)$, we call
$\ell$ a {\em critical leaf}; the image of a critical leaf is thus
degenerate, by definition. Let $\lam^{\ast}=\cup_{\ell\in\lam}\ell$ and
$\sigma_d^{\ast}:\lam^{\ast}\rightarrow\overline{\D}$ be the linear extension
of $\sigma_d$ over all the leaves in $\lam$. It is not hard to check that
$\sigma_d^{\ast}$ is continuous. Also, $\sigma_d$ is locally one-to-one on
$\mathbb{S}$, and $\sigma_d^{\ast}$ is one-to-one on any given non-critical
leaf.  Note that if $\lam$ is a lamination (which includes all points of
$\uc$ as degenerate leaves), then $\lam^{\ast}$ is a continuum. For
simplicity in what follows we often use the notation $\si_d$ for $\si_d^*$.

\begin{definition}[Gap]\label{gap-dfn}
A {\em gap} $G$ of a lamination $\lam$ is the closure of a component of
$\D\sm\lam^{\ast}$; its boundary leaves are called \emph{edges (of the gap)}.
Also, given a closed subset of $\uc$ we will call its convex hull a
\emph{gap}, too, even in the absence of a lamination.
\end{definition}

For each set $A\subset \overline{\D}$, denote $A \cap \uc$ by $\V(A)$
 and call the elements of $\V(A)$ \emph{vertices} of $G$.
If $G$ is a leaf or a gap of $\lam$, then $G$ coincides with the convex hull of
$\V(G)$.
A gap $G$ is called \emph{infinite (finite)} if and only if $\V(G)$ is
infinite (finite). A gap $G$ is called a \emph{triangular gap} (or, simply, a
\emph{triangle}) if $\V(G)$ consists of three points. Infinite gaps $G$,
with uncountable $V(G)$, are also called \emph{Fatou} gaps. Given points $a,
b\in\partial G$, let $(a, b)_G$ be the positively oriented open arc in
$\partial G$ from $a$ to $b$.

The so-called \emph{barycentric} construction (due to Thurston \cite{thu85})
yields a further extension $\bar{\si}_d$ of $\si_d$ onto the entire closed
disk $\cdisk$ such that $\bar{\si}_d(G)$ equals the convex hull of
$\si_d(\V(G))$ (the map $\bar{\si}_d$ sends barycenters to barycenters and
then extends linearly on segments connecting barycenters to the boundaries).
Again, for simplicity in what follows we often use notation $\si_d$ for
$\bar{\si}_d$.

\begin{definition}\label{d:gen-lam}
Let $\lam$ be a lamination. The equivalence relation $\sim_\lam$ induced by
$\lam$ is defined by declaring that $x\sim_\lam y$ if and only if there
exists a finite concatenation of leaves of $\lam$ joining x to y.
\end{definition}

\begin{definition}[q-lamination] A lamination $\lam$ is called a
\emph{q-lamination} if the equivalence relation $\sim_\lam$ is laminational
and $\lam$ consists of the edges of the convex hulls of $\sim_\lam$-classes
(called \emph{$\sim_\lam$-sets} or \emph{$\lam$-sets}).
\end{definition}{}

\begin{remark}
Since a q-lamination $\lam$ consists of 
the edges of the convex hulls of $\sim_\lam$-classes, if two leaves of $\lam$
share an endpoint, they must be adjacent edges of a common finite gap. It
follows
that no more than two leaves of a q-lamination can share an 
endpoint.
\end{remark}

\begin{definition}[Invariant (pre)laminations]\label{inv-lam} A (pre)lamination
$\lam$ is \emph{($\sigma_d$-)in\-va\-ri\-ant} if,

\noindent (D1) $\lam$ is \emph{forward invariant}: for each $\ell \in \lam$
either $\sigma_d(\ell) \in \lam$ or $\sigma_d(\ell)$ is a point in $\uc$, and

\noindent (D2) $\lam$ is \emph{backward invariant}:
\begin{enumerate}
    \item For each $\ell \in \lam$ there exists a leaf $\ell' \in \lam$
        such that $\sigma_d(\ell')=\ell$.
    \item  For each $\ell \in \lam$ such that $\sigma_d(\ell)$ is a
        non-degenerate  leaf, there exist $d$ \textbf{disjoint} leaves
        $\ell_1,......\ell_d$ in $\lam$ such that $\ell = \ell_1$ and
        $\sigma_d(\ell_i)=\sigma_d(\ell)$ for all $i$.
\end{enumerate}{}
\end{definition}{}

\begin{definition}[Siblings]\label{siblings} Two chords are called
\emph{siblings}  if they have the same image. Any $d$ disjoint chords with
the same non-degenerate image are called a \emph{sibling collection}.
\end{definition}{}

Definition \ref{siblings} deals with chords and does not assume the existence
of any lamination at all; the definition itself does not require iterations.

\begin{definition}[Monotone Map]
\label{d:monotone} Let $X$, $Y$ be topological spaces and $f:X\rightarrow Y$
be continuous. Then $f$ is said to be {\em monotone} if $f^{-1}(y)$ is
connected for each $y \in Y$.
\end{definition}

It is known that if $f$ is monotone and $X$ is a continuum then $f^{-1}(Z)$
is connected for every connected $Z\subset f(X)$.

\begin{definition}[Gap-invariance] \label{d:gap-inv} A lamination $\lam$ is
\emph{gap invariant} if for each gap $G$, its image $\sigma_d(G)$ is either a
gap of  $\lam$, or a leaf of  $\lam$, or a single point. In the first case, we
also require that $\sigma_d$ can be extended continuously to the boundary of
$G$ as a composition of a monotone map and a covering map onto the boundary of
the image gap, with positive orientation.
In other words, as you move through the vertices of $G$ in clockwise direction around $\Bd G$, their corresponding
images in $\sigma_d(G)$ must also move clockwise in $\Bd\sigma_d(G)$.
\end{definition}{}

\begin{definition}[Degree]
Suppose that both $G$ and $\si_d(G)$ are gaps.
The topological  degree of the extension of $\sigma_d$ to 
$\partial G$ is called the \emph{degree} of $G$. In other words, if every leaf of
$\sigma_d(G)$, except, possibly, for finitely many leaves, has $k$ preimage
leaves in $G$, then the degree of the gap is $k$.
A gap $G$ is called a \emph{critical} gap if either $k>1$, or $\si_d(G)$ is not a gap (a leaf or a point).
\end{definition}

The next two results are proved in \cite{bmov13}.

\begin{theorem}
Every $\sigma_d$-invariant lamination is gap invariant.
\end{theorem}{}

\begin{theorem}\label{t:prelam-cl}
The closure of an invariant prelamination is an invariant lamination. The
space of all $\sigma_d$-invariant laminations is compact.
\end{theorem}{}


It is convenient to consider some objects that normally come with a
lamination (e.g.,   gaps), as ``stand alone'' objects. Given the convex hull
$G$ of a closed set $T\subset \uc$ we define $\si_d(G)$ to be the convex hull of
$\si_d(T)$. This allows us to define the sets $\si_d^n(G)$ for all $n\ge 0$.

\begin{definition}\label{d:stand-alone-gaps}
A convex hull $G$ of a closed set $T\subset \uc$ is said to be a \emph{stand
alone} gap (of $\si_d$) if the following holds.

\noindent $(1)$ No chord in $\si_d^i(G)$ crosses a chord in $\si_d^j(G)$ for $i\ne  j$.

\noindent $(2)$ For every $i$, if the set $\sigma^i_d(G)$ has non-empty
interior, then we require that $\si_d|_{\partial \si^{i-1}_d(G)}$ can be
represented as a composition of a monotone map and a covering map onto the
boundary of $\si^{i}_d(G)$, with positive orientation. In other words, as you
move through the vertices of $\si_d^{i-1}(G)$ in clockwise direction around
$\partial \si_d^{i-1}(G)$, their corresponding images in $\sigma^i_d(G)$ must
also move clockwise.
\end{definition}

\subsection{Specific properties of general
invariant laminations}\label{ss:specific}

Here are basic definitions concerning periodic and preperiodic leaves/gaps.

\begin{definition}[Preperiodic  points] A point $x\in \uc$ is said to be
\emph{preperiodic} if $\sigma_3^{m+k}(x)=\sigma_3^m(x)$ for some $m
\geq 0 , k\geq 1$. The smallest $m$ and $k$ that satisfy the above
equation are called the \emph{preperiod} and  the \emph{period} of $x$,
respectively. A preperiodic point $x$ is either \emph{strictly
preperiodic} if $m>0$, or periodic (of \emph{period} $k$) if $m=0$.

\smallskip

\noindent (1) \emph{Preperiodic leaves.} Let $\ell$ be a leaf of a
    \emph{cubic lamination} $\lam$. The leaf $\ell$ is
    \emph{preperiodic (of preperiod $m$ and period $k$)}, if the endpoints $a$ and
    $b$ of $\ell$ are preperiodic of preperiod $m$ and (minimal) period $k$ or $2k$
    (in the latter case, $a$ and $b$ are required to lie in the same cycle).
    The leaf $\ell$ is strictly preperiodic if $m>0$, or periodic if $m=0$.

\smallskip

\noindent (2) \textit{Preperiodic gaps}. Let $G$ be a gap of a cubic
    lamination $\lam$. The gap $G$ is said to be \emph{preperiodic} if
    $\sigma_3^{m+k}(G)=\sigma_3^m(G)$ for some $m \geq 0 , k\geq 1$. The
    smallest $m$ and $k$ that satisfy the above equation are called the
    \emph{preperiod} and the \emph{period} of $G$, respectively. The gap $G$ is either
    preperiodic if $m>0$, or periodic if $m=0$. A periodic gap of period 1
    is also called \emph{fixed} or \emph{invariant}.

\smallskip

\noindent (3) \textit{Precritical gaps}. Similarly, we say that $G$ is
a
    \emph{precritical} gap, if $\sigma_3^k(G)$ is critical gap for some
    $k\geq0$.
\end{definition}

We will also need Theorem \ref{t:gaps-1}.

\begin{theorem}[\protect{\cite[Lemma 2.31]{bopt20}}]\label{t:gaps-1}
Let $G$ be an infinite periodic gap of period $n$ and set $K=\partial G$.
Then $\si_d^n|_K:K\to K$ is the composition of a covering map and a monotone
map of $K$. If $\si_d^n|_K$ is of degree one, then either statement {\rm (1)}
or statement {\rm (2)} below holds.

\begin{enumerate}
\item The gap $G$ has countably many vertices, finitely many of which are
    periodic of the same period, and the rest are preperiodic. All
    non-periodic edges of $G$ are $($pre$)$critical and isolated. There is
    a critical edge with a periodic endpoint among the edges of gaps from
    the orbit of $G$.

\item The map $\si_d^n|_K$ is monotonically semi-conjugate to an irrational
    circle rotation so that each fiber of this semiconjugacy is a finite
    concatenation of $($pre$)$critical edges of $G$. Thus, there are
    critical leaves (edges of some images of $G$) with non-preperiodic
    endpoints.
\end{enumerate}

In particular, if all critical sets of a lamination are non-degenerate finite
polygons then the lamination has no infinite gaps.
\end{theorem}

\begin{proof}
All claims of the theorem are proven in Lemma 2.31 \cite{bopt20}, except for
the last claim of (1), and the last claim of the entire lemma. The first of
these claims is about the existence of a critical edge with a periodic
endpoint among edges of gaps from the orbit of $G$. We may assume that $G$ is
invariant. Consider $\si_d|_{\partial G}$. This is a degree one map of the
Jordan curve of rational rotation number, and well-known properties of such
maps imply that it has at least one periodic point attracting from one side.
Since $\si_d$ is expanding on $\uc$, then there is a critical edge of $G$
with a periodic endpoint as claimed.

Let us prove the last claim of the lemma. Suppose that all critical sets of
$\lam$ are non-degenerate finite polygons, and yet $U$ is an infinite gap of
$\lam$. By Theorem \ref{t:kiwi} we may assume that $U$ is $n$-periodic. If
$\si_3^n|_{\partial U}$ is of degree greater than $1$ then for some $i$ we
must have $\si_3|_{\partial \si_3^i(U)}$ $k$-to-$1$ with $k>1$, a
contradiction with the assumption that all critical sets of $\lam$ are
finite. Now, suppose that $\si_3^n|_{\partial U}$ is of degree one. Then by
(1) and (2) there exists a critical leaf, a contradiction.
\end{proof}

A chord $\ell$ in a gap $G$ is a \emph{diagonal} of $G$ if
$\ell\not\subset\partial G$. The gaps described in Theorem \ref{t:gaps-1} are
called \emph{caterpillar} in case (1) and \emph{Siegel} in case (2).

\begin{Lemma}\label{l:adiag-1}
If $G$ is a gap such that $\si_d|_G$ is of degree one, $\ell=\ol{ab}$ is a
diagonal of $G$, and $\ell$ does not share a vertex with a critical edge of
$G$, then $\si_d(\ell)$ is a diagonal of $\si_d(G)$. A diagonal of a Siegel
gap eventually collapses to a point or has crossing images. A diagonal of a
caterpillar gap such that its iterated images are disjoint from critical
leaves will eventually map to a periodic diagonal.
\end{Lemma}

\begin{proof}
Since $\si_d|_G$ is of degree one, $\si_d(\ell)$ is not a diagonal of
$\si_d(G)$ only if the arc, say, $[a, b]_G=I$ collapses onto $\si_d(\ell)$.
Thus, $I$ is a finite concatenation of edges $\ell_1=\ol{ax},$ $\dots$ of
$G$, and the endpoints of the edges map to $\si_d(a)$ or to $\si_d(b)$. If
$\si_d(x)\ne \si_d(a)$ then $\si_d(x)=\si_d(b)$; since $\si_d|_G$ is of
degree one, then all remaining edges of $G$ contained in $I$ are critical.
So, if $\ell$ does not share a vertex with a critical edge of $G$, then this
is impossible and $\si_3(\ell)$ remains a diagonal of $\si_3(G)$.

Suppose that $G$ is a Siegel gap of period $j$. Assume that $X$ is a diagonal
of $G$ that never collapses to a point. Then the monotone map that collapses
edges of $G$ to points and semi-conjugates $\si^j_3|_{\partial G}$ to an
irrational rotation $\xi$ will project $X$ to a non-degenerate chord $\ell$
of $\uc$ (otherwise $X$ connects points connected by a finite concatenation
of a few precritical edges of $G$ which implies that $X$ does eventually
collapse to a point). This yields that there is an iterate of $\xi$ under
which $\ell$ maps to a chord that crosses $\ell$, implying that some iterated
images of $X$ cross.

Consider now a diagonal $X$ of a caterpillar gap $G$ such that the iterated
images of $X$ are disjoint from critical leaves. Then by the first claim of
the lemma all images of $X$ remain diagonals of the corresponding images of
$G$. The claim now follows from Theorem \ref{t:gaps-1}.
\end{proof}

From now on $\lam$ denotes a \emph{cubic} (i.e., $\si_3$-invariant)
lamination. A leaf $\ell$ is \emph{invariant (under $\sigma_3^k$)} if
$\sigma_3^k(\ell)=\ell$. The 4 invariant leaves of $\si_3$ are $\ol{0
\frac12},$ $\ol{\frac14 \frac34},$ $\ol{\frac18 \frac38}$ and $\ol{\frac58
\frac78}$. (Here, we use the identification between $\uc$ and $\R/\Z$, so
that, for example, $\ol{0\frac 12}$
 is the horizontal diameter.)
The first leaf has fixed
endpoints, the other three flip under the action of $\si_3$, and only $\ol{0
\frac12}$ and $\ol{\frac14 \frac34}$ contain the center 
of $\uc$.

Define the length $\|\ol{ab}\|$ of a chord $\ol{ab}$ as the shorter of the
lengths of the arcs in $\uc=\R/\Z$ with the endpoints $a$ and $b$. The
maximum length of a chord is $\frac12$. We divide leaves into three
categories by their length. 

\begin{definition}A \textit{short leaf} is a leaf $\ell$ such that $0<\|\ell\|<\frac{1}{6}$,\\
a \textit{medium leaf} is a leaf $\ell$ such that $\frac{1}{6}\leq\|\ell\|<\frac{1}{3}$ and \\
a \textit{long leaf} is a leaf $\ell$ such that
$\frac{1}{3}<\|\ell\|\leq\frac{1}{2}$.
\end{definition}

\emph{Critical} leaves are leaves of length $\frac13$. Let $\gamma(t)$ be the
distance from $t\in \R$ to the nearest integer. Call $\Gamma(t)=\gamma(3t)$
the \emph{length} function.

\begin{remark}\label{length-fn}

(1) For any leaf $\ell$, we have $\|\si_3(\ell)\|=\Gamma(\|\ell\|)$.


\noindent (2) If $0<\|\ell\|<\frac{1}{4}$, then
    $\|\sigma_3(\ell)\|>\|\ell\|$; if $\|\ell\|=\frac{1}{4}$, then
    $\|\sigma_3(\ell)\|=\|\ell\|$; if $\frac{1}{4}<\|\ell\|<\frac{1}{2}$, then
    $\|\sigma_3(\ell)\|<\|\ell\|$; if $\|\ell\|=\frac{1}{2}$, then
    $\|\sigma_3(\ell)\|=\|\ell\|$.

\noindent (3) For leaves of length bigger than $\frac{1}{4}$, the closer the
leaves  get to a critical chord (of length $\frac{1}{3}$) of the circle, the
shorter their images get.

\noindent (4) For a non-degenerate chord $\ell$, there is $n\ge 0$ such
that $\|\si_3^n(\ell)\|\ge \frac14$.

\end{remark}

\begin{definition}
A leaf $\ell$ is \emph{closer to criticality} than a leaf $\ell'$ if
$\|\ell\|$ is closer to $\frac 13$ than $\|\ell'\|$. This naturally defines
leaves closest to criticality in a specified family of leaves (observe that
in closed families, closest leaves must exist, yet if a family of leaves is
not closed, then its closest leaf does not have to exist).
\end{definition}

\begin{Lemma}\label{l:closest}
Any point $x\in (0, \frac12)$ that does not eventually map to $\frac14$ under $\Gamma$,
 eventually maps to $(\frac14, \frac{5}{12})$.
For a $\Gamma$-periodic but non-fixed point $t$ the closest to $\frac 13$ iterated $\Gamma$-image of $t$
belongs to $(\frac14, \frac{5}{12})$.
\end{Lemma}

\begin{proof}
Clearly, $\Gamma(\frac{5}{12})=\Gamma(\frac14)=\frac14$. If $x\in
(\frac5{12}, \frac12)$ then $x$ will eventually map into $(\frac14,
\frac{5}{12})$. Now, if $x\in (0, \frac14)$ then $x$ will be eventually
mapped to $(\frac14, \frac12]$ which, by the previous sentence, implies the
desired.
\end{proof}

\begin{Lemma}\label{l:diameter}
For a lamination $\lam$, exactly one of the following holds:
\begin{enumerate}
  \item chords $\ol{0 \frac12}$, $\ol{\frac16 \frac13}$, $\ol{\frac23 \frac56}$ are leaves of $\lam$;
  \item chords $\ol{\frac14 \frac34}$, $\ol{\frac{11}{12}\frac{1}{12}}$, $\ol{\frac{5}{12} \frac{7}{12}}$ are leaves of $\lam$;
  \item no leaves of $\lam$ have length $\frac12$ or $\frac16$.
\end{enumerate}
\end{Lemma}

\begin{proof}
If $\lam$ has a leaf $\ell$ of length $\frac12$, then $\si_3(\ell)$ is of length $\frac12$.
Since $\ell$ and $\si_3(\ell)$ must not cross, we see that $\si_3(\ell)=\ell$.
Thus, either $\ell=\ol{0 \frac12}$ (which by properties of laminations forces
leaves $\ol{\frac16 \frac13}$ and $\ol{\frac23 \frac56}$), or
$\ell=\ol{\frac14 \frac34}$ (which forces leaves $\ol{\frac{11}{12}
\frac{1}{12}}, \ol{\frac{5}{12} \frac{7}{12}}$). This completes the proof.
\end{proof}

\section{Symmetric cubic laminations}


\subsection{Odd cubic polynomials} 
Let $o$ be the origin in $\C$; this is the point $(0,0)$ in the Cartesian 
coordinate system. 
We write $o$ instead of $0$ in order not to confuse $o$ with
 the point of the unit circle whose argument is $0$.

A cubic polynomial $f$ is \emph{odd} if $f(-z)=-f(z)$. If $f$ is odd, then
$f(z)=az^3+bz$ is linearly conjugate to a polynomial
$P_\lambda(z)=z^3+\lambda z$. Assume that $J(P_\la)$ is connected and locally
connected and consider the $\sigma_3$-invariant laminational equivalence
relation $\sim_{P_\la}=\sim$ (see Section~\ref{ss:ler}). On top of satisfying
the axioms (I1) - (I3) of Definition~\ref{d:si-inv-lam}, the laminational
equivalence relation $\sim$ is such that for any $\sim$-class $\mathbf{g}$
the set $\mathbf{-g}$ is a $\sim$-class. Thus, the relation $\sim$ is
invariant with respect to the rotation $\tau$ by $180^{\degree}$ about $o$,
and $\tau(A)=-A$ (we use both notations interchangeably). We study all
$\sigma_3$-invariant laminations that satisfy this additional property.
Observe that $\si_3$ and $\tau$ commute. Recall that we identify the unit circle
$\uc$ with $\R/\Z$ and parameterize it as $[0,1)$. In this
parametrization, the coordinates of the two endpoints of a diameter of $\uc$
differ by $\frac{1}{2}$, and the endpoints of $\ell$ and $-\ell$ differ by
$\frac{1}{2}$, too.

\begin{definition}[Symmetric laminations]\label{cs-lam} A $\sigma_3$-invariant
lamination $\lam$ is called a \emph{symmetric (cubic) lamination} if (D3)
$\ell \in \lam$ implies $-\ell\in \lam$.
\end{definition}{}

From now on by a ``symmetric lamination'' we mean a ``symmetric cubic
lamination'', by a ``symmetric set'' we mean a $\tau$-invariant set, and by
$\lam$ we mean a symmetric cubic lamination.

\subsection{Symmetric laminations: basic properties}



\begin{definition} A \emph{central symmetric gap/leaf $G$} is a $\tau$-invariant
gap/leaf $G$; evidently, such a set $G$ contains $o$ (in its interior if $G$ is a gap).
\end{definition}



\begin{Lemma}\label{cs-gap}
The following holds.
\begin{enumerate}
  \item There is an invariant central symmetric gap or leaf $\cg(\lam)$ of $\lam$ containing $o$.
If $\cg(\lam)$ is a leaf, then case $(1)$ or $(2)$ of Lemma \ref{l:diameter} holds.
If $\cg(\lam)$ is a gap, then two symmetric edges of $\cg(\lam)$ have length $\ge\frac13$
 while the other edges have length $<\frac16$.
  \item There are exactly two distinct critical sets of $\lam$.
  \item Rotating $\lam$ by $90^{\degree}$ about $o$ results in another
      invariant symmetric lamination.
\end{enumerate}
\end{Lemma}

\begin{proof}
(1) Set aside cases (1) and (2) of Lemma \ref{l:diameter}. Choose a gap $G$
containing $o$ it its interior. Clearly $\tau(G)$ is a gap of $\lam$
too, thus $\tau(G)=G$, and $G$ contains diagonal diameters. It follows that
$\si_3(G)$ also has diagonal diameters and hence $\si_3(G)=G$. Consider now
how the length of circle arcs that are components of $\uc\sm G$ changes as we
apply $\si_3$. If an arc $(a, b)$ like that is of length less than $\frac13$,
then it maps one-to-one onto the arc $(\si_3(a), \si_3(b))$ and its length
triples. Hence, 
there exists an arc-component of $\uc\sm G$ of
length greater than or equal to $\frac13$. The rest easily follows.

(2) By (1), two symmetric circle arcs of length at least $\frac13$ are subtended by
edges of $\cg(\lam)$. They must contain two distinct critical sets of $\lam$.

(3) This claim follows from Definition \ref{inv-lam} and property (D3) of
Definition \ref{cs-lam} (that is, symmetry of $\lam$).
\end{proof}

From now on $\cg(\lam)$ denotes the invariant central symmetric gap of
$\lam$. Let $M=\ol{ab}$ be an edge of $\cg(\lam)$ of length $\|M\|\ge
\frac13$ (see Lemma \ref{cs-gap}) and the circle arc $H=(a, b)$
contains no vertices of $\cg(\lam)$. A sibling $M'$ of $M$ with endpoints
in $H$ is medium ($M$ is long). Finally, observe that, by Definition
\ref{siblings}, a non-critical leaf $\ell$ of $\lam$ has 2 siblings, and two
sibling leaves of the same kind have the same length. Lemma \ref{l:sml} is
straightforward and left to the reader.

\begin{Lemma}\label{l:sml}
The possibilities for chords in a sibling collection are
\begin{description}
    \item[(sss)] all chords are \textit{short};
    \item[(mmm)] all chords  are \textit{medium};
    \item[(sml)] one leaf is \textit{short}, one \textit{medium}, and one
        \textit{long}.
\end{description}
A sibling collection is completely determined by its type and one leaf.
\end{Lemma}

If a sibling collection has a long leaf, the collection is of type (sml).
Sibling collections of type (sss) of (mmm) partition the disk into 4
components (a ``central'' one and three ``side'' ones that can all be
obtained from each other by rotations by $\frac13$ and $\frac23$) while
collections of type (sml) partition the disk into three components with no
rotational symmetry.

\begin{definition}\label{d:sibli}
Suppose that $\ell=\ol{ab}$ is a non-critical chord which is not a diameter
and the arc $(a, b)$ is shorter than the arc $(b, a)$. Denote the chord
$\ol{(a+\frac13) (b-\frac13)}$ by $\ell'$ and the chord $\ol{(a+\frac23)
(b-\frac23)}$ by $\ell''$.
\end{definition}

As $\si_3(\ell')=\si_3(\ell'')=\si_3(\ell)$, $\{\ell, \ell', \ell''\}$ is a
sibling collection. For a long/medium non-critical 
chord $\ell$ it follows that $\ell'$ is long/medium and $\ell''$ is small;
if, moreover, $\ell\in \lam$ (recall that $\lam$ is a symmetric lamination),
its sibling collection is $\{\ell, \ell', \ell''\}$ (all other possibilities
lead to crossings with $\ell$ or $-\ell$). So, a sibling collection of type
(mmm) is impossible.

\begin{definition}\label{d:strip} Let $\ell$ and $\ell'$ be two disjoint chords of $\uc$.
Consider the component of $\cdisk\sm [\ell\cup\ell']$ between $\ell$ and $\ell'$.
The closure $\mathcal{S}(\ell,\ell')$ of this component is called the
\emph{strip between the chords $\ell$ and $\ell'$}.
The strip $\mathcal{S}(\ell,\ell')$ is bounded by the leaves $\ell$
and $\ell'$ and two arcs of $\uc$; define the \emph{width} of the strip
$\mathcal{S}(\ell,\ell')$ to be the length of the larger of those two arcs.
\end{definition}

\begin{definition} [Short strips]
For a sibling collection $\{\ell,\ell', \ell''\}$ of type (sml), with $\ell$
 and $\ell'$ long/medium, set $C(\ell)=\mathcal{S}(\ell,\ell')$.
The set $C(\ell)$ has width $|\frac13-\|\ell\| |$ (and so does $-C(\ell)$).
Given a long/medium chord $\ell\in \lam$, call the region
$\sh(\ell)=C(\ell)\cup -C(\ell)$ the \emph{short strips (of $\ell$)} and each
of $C(\ell)$ and $-C(\ell)$ a \emph{short strip (of $\ell$).}
The width of $C(\ell)$ will also be referred to as the width of $\sh(\ell)$).
Note that $-C(\ell)=C(-\ell)$.
\end{definition}

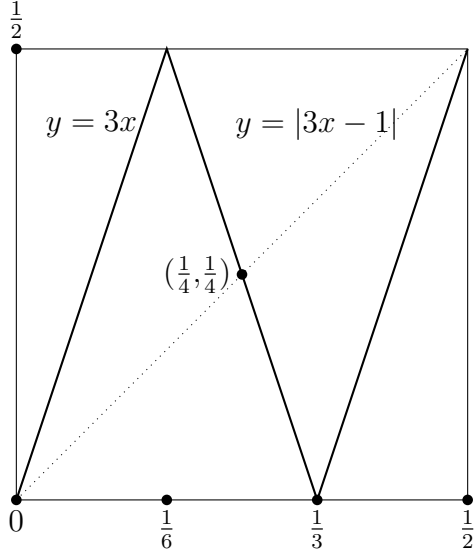
\begin{figure}[ht]
\centering
    \begin{tikzpicture}
\draw (0,0) -- (6,0) -- (6,6) -- (0,6) -- (0,0);
\draw [thick](0,0) -- (2,6) -- (4,0) -- (6,6);
\draw [dotted](0,0) -- (6,6);
          \fill[black] (0,0) circle[radius=2pt] ++(0:0em) node [below]{0};
          \fill[black] (6,0) circle[radius=2pt] ++(0:0em) node [below]{$\frac{1}{2}$};
          \fill[black] (0,6) circle[radius=2pt] ++(0:0em) node [above]{$\frac{1}{2}$};
          \fill[black] (2,0) circle[radius=2pt] ++(0:0em) node [below]{$\frac{1}{6}$};
          \fill[black] (4,0) circle[radius=2pt] ++(0:0em) node [below]{$\frac{1}{3}$};
          \fill[black] (3,3) circle[radius=2pt] ++(0:0em) node [left]{($\frac{1}{4}$,$\frac{1}{4}$)};
          \node (A) at (4,5) {$y=|3x-1|$};
          \node (B) at (7/2,3/2) {};
          \node (C) at (9/2,3/2) {};
           \node (D) at (1,5) {$y=3x$};
           \node (E) at (1/2,3/2) {};

\end{tikzpicture}
    \caption{Graph of the length function $\Gamma$. Length $\|\ell\|$ of a leaf $\ell$
    is on $x$-axis and length $\Gamma(\|\ell\|)$ of the image leaf $\sigma_3(\ell)$ is on the $y$-axis. }

\end{figure}

Here are properties of short strips $\sh(\ell)$ of a long/medium non-critical
chord $\ell$.

(a) The short strip $C(\ell)$ is bounded by a chord $\ell$ and its sibling
$\ell'$ and the short strip $-C(\ell)$ is bounded by the chord $-\ell$ and
its sibling $-\ell'$. All these chords are long/medium.

(b) Any critical chord of $\uc$ that does not cross any of the four chords
$\{\ell, \ell',-\ell, -\ell'\}$ lies inside a short strip of $\ell$.

(c) Any chord or gap in the complement of $\sh(\ell)$ maps $1$-to-$1$ onto
its image.

(d) If $\lam$ is a symmetric lamination and $\ell\in \lam$ is long/medium,
then any leaf that is closer than $\ell$ to criticality is contained in
$\sh(\ell)$.

(e) For two leaves of $\lam$, their short strips, if exist, are nested.

The next lemma will be applied to leaves of laminations or in similar cases.
However, it holds for any chords.

\begin{Lemma}[\textbf{Short Strip Lemma}]\label{l:str}
Let $\ell=\ell_0$ be a chord, and set $L=\|\ell\|>\frac16$,
$\ell_i=\sigma_3^i(\ell)$, $L_i=\|\ell_i\|$.
Take the minimal positive integer $k$ such that $\ell_k$ intersects the interior of $\sh(\ell)$.

\begin{enumerate}
  \item We have $L_k>w(C(\ell))$.
  If $\ell_k$ does not cross the edges of $\sh(\ell)$, then $\ell_k$ is closer to criticality than $\ell$
  (and so $\ell_k$ is long/medium).
  \item If $L=\frac14$, and $\ell$ is a leaf of a cubic symmetric lamination $\lam$, then either
  $\ell\in\{\ol{\frac18 \frac38},\ol{\frac58 \frac78}\}\subset \lam$,
  or $\ell\in\{\ol{\frac78 \frac 18},\ol{\frac38 \frac58}\}\subset \lam$. 
  \item If $L>\frac14$, then $L_k>3w(C(\ell))$.
 \item If a chord $\ell$ is the closest to criticality in its forward orbit,
  then $\ell$ is long/medium, and no forward image of $\ell$ enters the interior of $\sh(\ell)$.
\end{enumerate}
\end{Lemma}

\begin{proof} (1) The leaf $\ell$ and its sibling $\ell'$ form a part of the boundary of
$C(\ell)$. 
Note that $w(C(\ell))=|\frac13-L|=t<\frac16$. We claim that $L_k>t$. Otherwise, choose
the least $j$ with $L_j\le t$; then $0<j$ (because $L>\frac16>t$) and $j\le k$
by the assumption. By the properties of $\Ga$, either 
$L_{j-1}=\frac{L_j}{3}$, or 
$L_{j-1}=\frac13\pm\frac{L_j}{3}$.
By the choice of $j$, 
 the former is impossible.
Now, if 
 the latter holds, then
$|L_{j-1}-\frac13|=\frac{L_j}{3}<t$, and so $\ell_{j-1}$ is closer to
criticality than $\ell$, a contradiction with $\ell_{j-1}$ being disjoint
from the interior of $\sh(\ell)$. Thus, $L_k>t$. The last claim of the lemma
is immediate.


(2) If $|\ell|=\frac14$, then the edges of $\sh(\ell)$ partition $\cdisk$ into
components so that the only two leaves of $\lam$ of length $\frac14$ are
$\ell$ and $-\ell$. Since $\|\si_3(\ell)\|=\frac14$, it follows that either
$\si_3(\ell)=\ell,$ $\si_3(-\ell)=-\ell$
 (then $\ell\in\{\ol{\frac18 \frac38},\ol{\frac58 \frac78}\}\subset \lam$) or
 $\si_3(\ell)=-\ell,$ $\si_3(-\ell)=\ell$
 (then $\ell\in\{\ol{\frac78 \frac 18},\ol{\frac38 \frac58}\}\subset \lam$).

(3) The argument is similar to (1), with one difference. In (1), we find a
moment before $k$ such that the length of the chord \emph{drops} to $t$ or
less. This works out because $L>\frac16>t$ and hence the desired
moment is not $0$. 
To prove (3) it suffices to
observe that since now $L>\frac14$ then $L>3w(C(\ell))=1-3L$; hence,
repeating the arguments from (1), but replacing in them $t$ by $3t$, we will
come to the same conclusion.

(4) By Lemma \ref{l:closest}, $\frac{5}{12}\ge |\ell\|\ge \frac14$. Now (1)
implies the desired.
\end{proof}

\begin{figure}[]
\centering
\includegraphics[width=0.6\linewidth, height=0.3\textheight]{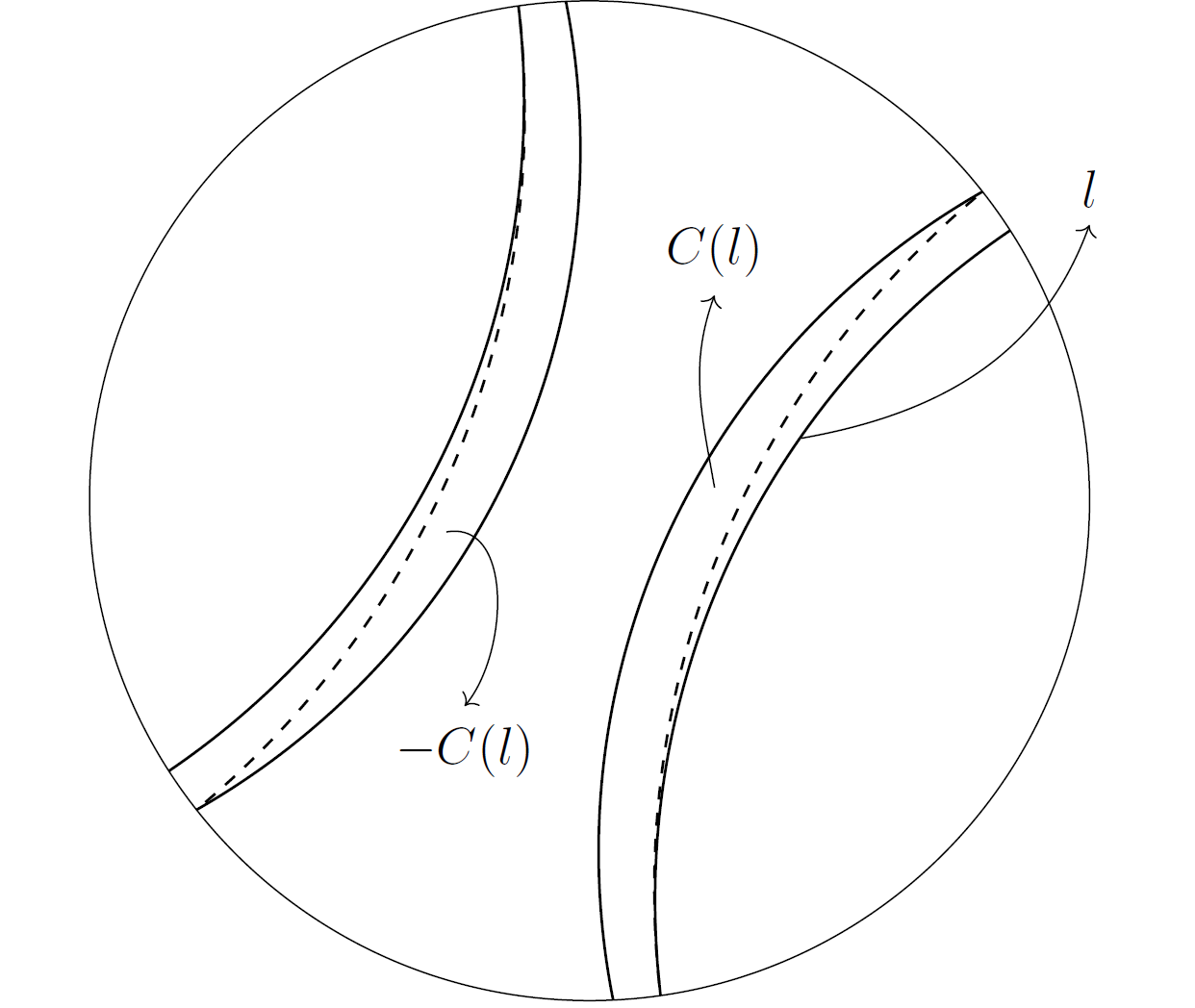}
    \caption{Short strips $C(l)\cup -C(l)$. Dashed line is a critical chord inside the strip.}
\end{figure}

For a gap, by \emph{collapsing} we mean mapping to a leaf or a point. In the
case of symmetric laminations, by Lemma \ref{cs-gap} there are two distinct
critical sets of $\lam$, hence collapsing to a point is impossible.

\begin{theorem}[\textbf{No Wandering Triangles}]\label{nwt-thm} Let $\lam$ be
a symmetric lamination and $G$ be a gap of $\lam$. If $G$ does not eventually
collapse, then $G$ is preperiodic.
\end{theorem}

\begin{proof}
We may assume that $G$ is a triangle. If $G$ is not preperiodic and never
collapses, $\{G_n=\sigma_3^n(G)\}_{n=0}^{\infty}$ is an infinite sequence of
gaps. Let $d_n$ be the length of the shortest edge of $G_n$; then $d_n>0$ and
$d_{n}\to 0$. Let $\ell_n$ be the longest edge of $G_n$. Define a sequence
$n_i+1$ of all times when $d_{n_i+1}$ is less than all previous $d_{n}$'s.
For large $i$, the gap $G_{n_{i+1}+1}$ has an edge of length $d_{n_{i+1}+1}$, the
image of $\ell_{n_{i+1}}$.
Since $d_{n_{i+1}+1}<d_{n_i+1}$, the leaf $\ell_{n_{i+1}}$
is closer to criticality than $\ell_{n_i}$. Hence $\ell_{n_{i+1}}$ is
contained in a short strip of $\ell_{n_i}$. However then $G_{n_{i+1}}$ has an
edge shorter than $w(\sh(\ell_{n_i}))=\frac13\cdot d_{n_i+1}$, a
contradiction with the choice of $n_{i+1}+1$.
\end{proof}

\section{Finite gaps}

Let us study finite gaps of symmetric laminations.

\begin{definition}[Major]\label{d:major} Let $G$ be a periodic gap of a symmetric
lamination $\lam$. The edges of images of $G$ that are the closest to
criticality among all such edges are called \emph{majors (of the orbit of
$G$)} (there might be more than one major). If $M$ is such a major, then, by
Lemma \ref{l:closest}, we have $\|M\|\ge\frac14$.
\end{definition}


We use majors to study finite gaps of $\lam$.
By 
Theorem \ref{nwt-thm}, any finite gap $G$ eventually collapses or maps to a
periodic gap $\widetilde{G}$.
Periodic gaps $G$ can be classified into two
kinds.

(1) \textit{Gaps with symmetric orbits:} 
$\sigma_3^k(G)=-G$ for some $k>0$.

(2) \textit{Gaps without symmetric orbits:} 
$G$ and $-G$ are in distinct orbits.

Call a finite periodic gap of $\lam$ 
a \emph{periodic polygon}. Let $G$ be
a periodic polygon of period greater than $1$ and $\widetilde{G}$ be an
eventual image of $G$ containing a major $P$ of the orbit of $G$. Consider
the central symmetric gap/diameter $\cg(\lam)$ of $\lam$. In the diameter
case let $M=\cg(\lam)$; in the gap case consider majors $M, -M$ of
$\cg(\lam)$. In any case, consider short strips $\sh(M)$ bounded by the
leaves $M$, $-M$ and their siblings $M'$, $-M'$. By Lemma \ref{cs-gap}, we
have $\|M\|\ge \frac13$. There are two sibling gaps of $\cg(\lam)$; let $A$
be the one with edge $M'$, and let $B$ be the one with edge $-M'$.

In addition to $M'$, the gap $A$ has an edge which is a sibling of $-M$.
Using notation from Definition \ref{d:sibli}, we denote it by $-M''$. A
straightforward computation shows that $\|-M''\|\le \frac16$ (e.g., we can
insert an artificial diameter-diagonal $D$ in $\cg(\lam)$ and observe that
the appropriate sibling of $D$ is contained in $A$ and has the length
$\frac16$ which is, for geometric reasons, greater than or equal to
$\|-M''\|$ as desired). Similarly, the gap $B$ has an edge $M''$ which is a
sibling of $M$, too, and $\|M''\|\le \frac16$. Since $\|P\|\ge \frac14$ by
Lemma \ref{l:closest}, then $\widetilde{G}$ is inside a short strip from
$\sh(M)$ as $P$ fits nowhere else in the disk without crossing edges of
$\cg(\lam),$ $A$ and $B$. In particular, there exists exactly one other
long/medium edge $Q$ of $\widetilde{G}$ (in addition to $P$). Observe that
either $\cg(\lam)$ has two majors, or $\cg(\lam)$ is a diameter.

\begin{definition}\label{d:fx-rtn}
Let $G$ be a periodic gap of minimal period $k$. Then $G$ is said to be a
\emph{fixed return gap (of minimal period $k$)} if any two distinct forward
images of $G$ under the map $\sigma_3^i$ with $0\le i<k$ have disjoint
interiors and all vertices of $G$ are fixed by $\sigma_3^k$.
\end{definition}

We need the following 
 result 
 of Jan Kiwi \cite{kiw02}.

\begin{theorem}
\label{t:kiwi}
Let $\lam$ be a $\si_d$-invariant lamination. Then any infinite gap of $\lam$
is (pre)periodic. For any finite periodic gap $G$ of $\lam$ its vertices
belong to at most $d-1$ distinct cycles except when $G$ is a fixed return
$d$-gon. In particular, a cubic lamination cannot have a fixed return $n$-gon
for $n>3$. Moreover, if all images of a $k$-gon $G$ with $k>d$ have at least
$d+1$ vertices then $G$ is preperiodic.
\end{theorem}

Let us now go back to a symmetric lamination $\lam$. Recall that given a
long/medium leaf $\ell\in \lam$, its sibling collection is $\{\ell, \ell',
\ell''\}$.

\begin{Lemma}\label{l:23}
A triangle $T$ of $\lam$ does not share an edge with any $\si_3^n(T)\ne T$.
No two fixed return triangles of $\lam$ share an edge. A fixed return
triangle with long/medium side $M$ cannot map to a triangle with an edge
$M'$, $-M$ or $-M'$.
\end{Lemma}

\begin{proof}
By way of contradiction, let $\si_3^n(T)\ne T$ share an edge with $T$.
Properties of laminations imply that $T$ has vertices, say, $a,$ $b,$ $c$,
where $\si_3^n(a)=b,$ $\si_3^n(b)=a$ and $\si_3^n$ rotates $T$ accordingly.
Then the orbit of $T$ falls apart into pairs of triangles and each pair is
rotated by $\si_3^n$. We may assume that $\ol{ab}$ is a major of the orbit of
$\ol{ab}$ and $T$ is contained in $\sh(\ol{ab})$. If $m$ is a short side of
$T$ then $m$ is contained (except perhaps for the endpoints) in the interior
of the short strips generated by the major of the orbit of $m$, a
contradiction with Lemma \ref{l:str}.

To prove the second claim of the lemma assume, by the above, that two fixed
return triangles sharing an edge do not belong to the same cycle. Then we can
put them into one quadrilateral and observe that the existence of such a
quadrilateral contradicts Theorem \ref{t:kiwi}.

Let us prove the last claim of the lemma. Let $T$ be a fixed return triangle
such that $\si_3^n(T)=-T$; then $\si_3^n(-T)=T$. If $\si_3^n$ maps edges of
$T$ not to their $\tau$-images, then $\si_3^n$ applied to $-T$ will produce
the same rotation of edges of $-T=\si_3^n(T)$. Since the second iteration of
a non-trivial rotation of vertices of a triangle can never be the identity,
 $\si_3^{2n}(T)=T$ and $\si_3^{2n}$ on $T$ is not the identity, a
contradiction with $T$ being fixed return. So, $\si_3^n$ maps edges of $T$ to
their $\tau$-images. 

It follows that for any edge $e$ of any triangle from the orbit of $T$ we
have $\si_3^n(e)=-e$. Among all iterated images of $T$ choose a triangle that
has an edge $\ell$  closest to criticality among all edges of triangles in
the orbit of $T$; assume that this triangle is $T$ itself. Denote by $m$ its
short edge, and then choose $k$ such that $\si_3^k(m)=M$ is a major of the
orbit of $m$. It follows that $M$ enters its short strips as a short leaf, a
contradiction. Hence no fixed return triangle $T$ can map onto $-T$.

Now, let $T$ be a fixed return triangle with a long/medium edge $M$. It
cannot eventually map to a triangle with an edge $M'$ as otherwise images of
these two triangles are periodic triangles from the same orbit that share an
edge $\si_3(M)=\si_3(M')$, a contradiction.
If now $\si_3^n(M)=-M$ or $-M'$,
 then $\si_3(T)$ has edge $\si_3(M)$ that under $\si_3^n$ maps to the
 triangle $T'$ with edge $\si_3(-M)=\si_3(-M')=-\si_3(M)$. There is also a
triangle $-\si_3(T)$ with the edge $-\si_3(M)$.
By the above, $\si_3(T)$
cannot eventually map to $-\si_3(T)$. We conclude that the triangles
$T'=\si_3^{n+1}(T)$ and $-\si_3(T)$ share an edge $-\si_3(M)$ and are,
therefore, two fixed return triangles sharing an edge.
By the above, this is
impossible which proves the last claim of the lemma.
\end{proof}

\begin{Lemma}\label{P-or-P}
Let $G$ be a periodic polygon. Then $(1)$ the gap $G$ is not fixed return,
and $(2)$ each edge of $G$ eventually maps to $P$ or $-P$ where $P$ is a
major of the orbit of $G$.
\end{Lemma}

\begin{proof}
(1) By Theorem \ref{t:kiwi}, the only possible fixed return gap of a cubic
lamination is a triangle. Assume that a fixed return triangle $T$ of $\lam$
has an edge $\ell$, a major of the orbit of $T$.
Let $m$ be the only short edge of $T$.
Let $M=\si_3^n(m)$ be a major of the orbit of $m$ and an edge of a
triangle $H\ne T$ from the orbit of $T$.
By Lemma \ref{l:str}, the edge $m$ is disjoint
from the interior of $\sh(M)$. Since $m$ is an edge of $T$, then $T$ cannot
be contained in $\sh(M)$.
By Lemma \ref{l:23}, the triangle $T$, being an eventual image of
$H$, cannot have $M,$ $-M$, $M',$ or $-M'$ as an edge.
Then $\ell$ is closer to criticality than $M$, a contradiction.

(2) Let $P$ be an edge of $G$. For an edge $\ell$ of $G$, let $\hell$ be an
eventual image of $\ell$ which is closest to criticality; by Lemma
\ref{l:closest}, the leaf $\hell$ is long/medium. If $\hat{\ell}\notin
\{P,-P\}$, then $G\subset \sh(\hat{\ell})$ as $P$ is contained in the
interior of $\sh(\hell)$. By Lemma \ref{l:str}, the gap $G$ is not contained
in the interior of $\sh(\hell)$; hence a boundary edge $\tell$ of
$\sh(\hell)$ is an edge of $G$. However then, since $G$ is not fixed return,
$\tell$ will have an eventual image non-disjoint from the interior of
$\sh(\hell)$, a contradiction with Lemma \ref{l:str}.
\end{proof}

\begin{Lemma}\label{fn-gaps}
Let $G$ be a periodic polygon of a symmetric lamination, and let $g$ be
the first return map of $G$. One of the following is true.

$(a)$ The first return map $g$ acts on the sides of $G$ transitively
    as a rational rotation.

$(b)$ The edges of $G$ form two disjoint periodic orbits, $g$ permutes the
    sides of $G$ transitively in each orbit, and $G$ eventually maps to the
    gap $-G$. If $\ell$ and $\ell'$ are two adjacent edges of $G$, then the
    leaf $\ell$ eventually maps to the edge $-\ell'$ of $-G$.
\end{Lemma}

\begin{proof}
(a) By Theorem \ref{t:kiwi} (or because every edge of $G$ passes through $P$
or $-P$), the vertices of $G$ form one/two periodic orbits under the map $g$.
If the orbit of $G$ is not symmetric, then it does not include $-P$. Hence
there is a unique orbit of vertices of $G$ and (a) holds.

(b) If the vertices are in two orbits, then, by (a), the gap $G$ has a symmetric
orbit, and the majors $P$ and $-P$ of the orbit of $G$ have distinct orbits.
If $\si_3^k(\ell)=-\ell$ for some $k$, then $\si_3^k(P)=-P$ (because
$\si_3^k$ preserves orientation), a contradiction. Hence $\ell$ never maps to
$-\ell$ and the last claim of the lemma follows because the two orbits of
vertices alternate on the boundary of $G$.
\end{proof}

\begin{definition}
If case (a) from Lemma \ref{fn-gaps} holds, we call a gap $G$ a
\emph{1-rotational gap}. If case (b) from Lemma \ref{fn-gaps}
holds we call such a gap \emph{a 2-rotational gap}.

\end{definition}

Below are the two important properties of preperiodic polygons. 

\begin{corollary}\label{no-diagonal}
If $G$ is a preperiodic polygon of a symmetric lamination such that $G$ is
not precritical (e.g., if $G$ is periodic), then no diagonal of $G$ can be a
leaf of a symmetric lamination.
\end{corollary}

\begin{proof}
Let $\ell$ be a diagonal of $G$. If $G$ is 1-rotational, an eventual image of
$\ell$ crosses $\ell$, and $\ell$ cannot be a leaf of any lamination. Let $G$
be 2-rotational. Then the only way $\ell$ can possibly be a leaf of a
lamination is if there are clockwise consecutive vertices $a, b, c, d$ of $G$
and $\ell=\ol{ac}$. If $\ell\in \lam'$ where the lamination $\lam'$ is
symmetric, then $\ol{\tau(a)\tau(c)}\in \lam'$. Yet, by Lemma \ref{fn-gaps}
an eventual image of $\ell$ is $\ol{\tau(b)\tau(d)}$ which crosses
$\ol{\tau(a) \tau(c)}$, a contradiction.
\end{proof}

\begin{corollary}\label{dis-fn-gaps}
Two distinct preperiodic polygons have disjoint sets of vertices, unless both
are strictly preperiodic, share a common edge that eventually maps to a
critical leaf, and eventually both map to the same periodic polygon.
\end{corollary}

\begin{proof}
If \emph{periodic} polygons $G$ and $G'$ share an edge $\ell$ or a vertex
$v$, then the union of the orbits of $G$ and $G'$ is a union of connected
components permuted by $\si_3$.
Let $X$ be the component of the union containing $G\cup G'$.
Let $\si_3^k$ be the minimal iterate of $\si_3$ that maps
$X$ back to itself. We claim that $\si_3^k(\hell)\ne \hell$ for any leaf
$\hell\subset X$. Indeed, assume that $\hell\subset X$ is an edge of a gap
$H$ such that $\si_3^k(\hell)=\hell$. Then $\si_3^k$ either fixes the vertices of
$H$ or flips $H$ to the other side of $\hell$ so that the first return map
$\si_3^{2k}$ of $H$ fixes the vertices of $H$. Since both possibilities
contradict Lemma \ref{P-or-P}, we see that $\si_3^k(\hell)\ne \hell$ for any
leaf $\hell\subset X$.

Recall that $\bar{\si}_3$ denotes the barycentric extension of $\si_3$ onto
the closed unit disk $\cdisk$. If a gap $H\subset X$ maps to itself by
$\bar{\si}_3^k$, then, by Lemma \ref{P-or-P}, $\bar{\si}_3^k$ rotates the
edges of $G$ and closures of components of $X\sm G$ attached to the egdes of
$G$ (``decorations''). For some $i>1$, the map $\bar{\si}_3^{ik}$ fixes the
edges of $H$ for the first time. It follows that $\bar{\si}_3^{ik}$ maps gaps
contained in decorations to themselves for the first time and fixes their
vertices, again a contradiction with Lemma \ref{P-or-P}. Hence no gap
$H\subset X$ maps to itself by $\bar{\si}_3^k$.

Since $X$ is locally connected, there exists $x\in X$ with
$\bar{\si}_3^k(x)=x$.
Since, by the previous paragraph, no leaf/gap contained
in $X$ maps to itself by $\bar{\si}_3^k$, then $x$ is a vertex of a gap
$H\subset X$. Let $Y$ be the union of leaves in $X$ with endpoint $x$. By
Corollary 3.7 \cite{bmov13}, the orientation is preserved on $Y\cap \uc$
under $\si^k_3$, and since $\si_3^k(x)=x$ then $\si_3^k|_{Y\cap \uc}$ is the
identity, again a contradiction with Lemma \ref{P-or-P}.

Thus, any preperiodic polygons $G$ and $G'$ sharing a vertex eventually map
to the same polygon. Preperiodic polygons sharing a vertex whose image
polygon is the same must share a critical leaf on their boundaries, see Lemma
3.11 in \cite{bmov13}. This completes the proof.
\end{proof}


Lemma \ref {col-quad} deals with gaps which eventually map onto
\emph{collapsing quadrilaterals}, i.e., quadrilaterals collapsed to a leaf by
$\sigma_3$.

\begin{Lemma}\label{col-quad} Let $\{G,-G\}$ be a pair of collapsing quadrilaterals
of $\lam$ and $s$ be the length of their shorter sides. Then any gap $H$ with
$\si_3^n(H)=\pm G$ is a quadrilateral with a pair of opposite edges of length
$s/3^n$ that map to short edges of $\pm G$.
\end{Lemma}

\begin{proof}
If $\ell$ is an edge of $H$ with $\si_3^n(\ell)$ being an edge of $G$ of
length $s$, then by Lemma \ref{l:str} all iterated images of $\ell$ are short
which implies the claimed.
\end{proof}

\begin{Lemma}\label{l:infinite-periodic}
An infinite critical gap of a symmetric lamination is periodic.
\end{Lemma}

\begin{proof}
Let $G$ be an infinite critical gap of a symmetric lamination $\lam$.
Since $G$ and $-G$ contain critical chords in their interiors (except for
their endpoints), $\lam$ has no critical leaves. Assume that eventual images
of $G$ are not equal $G$ or $-G$.
By Theorem \ref{t:kiwi}, the lamination $\lam$ has an
eventual image $H$ of $G$ which is infinite with $\si_3^n(H)=H$ and
$\si_3^n|_{\partial H}$ one-to-one for some $n>0$. Since the edges of periodic
gaps eventually map to critical or periodic edges (see, e.g., Lemma 2.28 of
\cite{bopt20}) and there are no critical leaves, we can find $k$ so that some
edges of $H$ are $\si_3^{kn}$-invariant.
It follows from the fact that $H$ is infinite,
 that $\si_3^{kn}|_{\partial H}$ has attracting points, a
contradiction with the expanding properties of $\si_3$. Hence, $G$ is either periodic or eventually maps to $-G$, in which case it is also periodic.
By Lemma \ref{cs-gap}, the two critical sets of $\lam$ are $G$ and $-G\ne G$.
\end{proof}

\section{Comajors and their properties}

In this section, we  work towards understanding the structure of the family
of symmetric laminations. Every symmetric lamination has three important
kinds of special leaves: majors, comajors, and minors.
Those leaves carry
 enough information to reconstruct the lamination. Formal definitions
are given below.

\subsection{Initial facts}  \emph{From now on $\lam$ denotes a symmetric
lamination.}

If $c$ is a short chord, then there are two long/medium chords with the same
image as $c$. We will denote them by $M_c$ and $M'_c$. Also, denote by $Q_c$
the convex hull of $M_c\cup M'_c$. This applies in the degenerate case, too:
if $c\in \uc$ is just a point, then $M_c=M'_c=Q_c$ is a critical leaf $\ell$
disjoint from $c$ such that $\si_3(c)=\si_3(M_c)$.

\begin{definition}
[Major] A leaf $M$ of $\lam$ 
closest to criticality
is called a \emph{major} of $\lam$.
\end{definition}

If $M$ is a major of $\lam$, then the long/medium sibling $M'$ of $M$ is also
a major of $\lam$, as well as the leaves $-M$ and $-M'$. Thus, a lamination has
either exactly 4 non-critical majors or 2 critical majors.

\begin{definition}[Comajor] The short siblings of  the major leaves of $\lam$
are called  \emph{comajors}; we also say that they form a \emph{comajor
pair}. If the major leaves of $\lam$ have a sibling of length $1/6$, then
this sibling is also called a \emph{comajor}. A pair of symmetric chords is
called a \emph{symmetric pair}. If the chords are degenerate, then their
symmetric pair is called \emph{degenerate}, too. 
\end{definition}

A symmetric lamination has a symmetric pair of comajors $\{c,-c\}$.

\begin{definition}
[Minor] Images of majors (or, equivalently, comajors) are called
\emph{minors} of a symmetric lamination. Similarly to comajors, every
symmetric lamination has two symmetric minors $\{m,-m\}$ .
\end{definition}

Critical majors of a lamination have degenerate siblings,
 hence we have degenerate comajors and minors in this case.
If majors $M$ and $-M$ are non-critical, then there is a critical gap, say, $G$
with edges $M$ and $M'$, and a critical gap $-G$ with edges $-M$ and $-M'$.

\begin{Lemma}\label{short-leaves}
Let $\{m,-m\}$ be the minors of $\lam$, and $\ell$ be a leaf of $\lam$. Then no
forward image of $\ell$ is shorter than $\min(\|\ell\|, \|m\|)$.
\end{Lemma}

\begin{proof}
Since majors are the closest to criticality leaves of $\lam$, the image of
any long/medium leaf of $\lam$ is no shorter than the minor. On the other
hand, the image of any short leaf is three times longer than the leaf itself.
The lemma follows from these observations.
\end{proof}

\begin{Lemma}\label{cmajor-end-points}
Let $c$ be a comajor of $\lam$.

\begin{enumerate}

\item If $c$ is non-degenerate, then one of the following holds:

\begin{enumerate}

\item the endpoints of $c$ are both strictly preperiodic with the same
    preperiod and period;

\item the endpoints of $c$ are both not preperiodic, and $c$ is
    approximated from both sides by leaves of $\lam$ that have no common
    endpoints with $c$.

\end{enumerate}

\item If $M_c$ is non-critical, then its endpoints are both periodic, or
    both strictly preperiodic with the same preperiod and period, or both
    not preperiodic.

\end{enumerate}

In particular, a non-degenerate comajor is not periodic.

\end{Lemma}

\begin{proof}
Set $M=M_c$. It follows from Lemma \ref{short-leaves} and the equality
$\|\si_3(c)\|=3\|c\|$ that $c$ is non-periodic. Since $c$ is non-degenerate,
the lamination $\lam$ has two symmetric critical gaps $G$, $-G$, and pairs of
majors $\{M, M'\}$ and $\{-M,-M'\}$ as edges of $G$ and $-G$, respectively.
Assume first that at least one endpoint of $c$ is preperiodic. Then, by Lemma
2.25 of \cite{bopt20}, both endpoints of $c$ are preperiodic and the period
of eventual images of the endpoints of $c$ is the same.

We claim that their preperiods are equal. Indeed, otherwise we may assume
that an eventual \emph{non-periodic} image $\ell=\ol{ab}$ of $c$ has an
$n$-periodic endpoint $a$ and the leaf $\si_3^n(\ell)=\ol{ad}$ is
$n$-periodic.
This means that $\si_3^i(\ol{ab})=\si_3^i(\ol{ad})$ for some minimal $i>0$.
It is easy to see that the only way this can
happen is as follows: there is a collapsing quadrilateral $Q$ which is the convex
hull of majors, say, $\{M, M'\}$, and forward images of the leaves $\ol{ab},
\ol{ad}$ are edges of $Q$.

We may assume that in fact $\ol{ab}, \ol{ad}$ themselves are edges of $Q$
(and so they have equal $\si_3$-images), $\ol{ad}$ is periodic, and $\ol{ab}$
is not.
By Lemma \ref{l:str}, the majors $M$ and $M'$ can never be mapped to the short
sides of $Q$. Hence we may assume that $M=\ol{ad}$ is periodic. However, by
the above assumption it is $\ol{ab}$ which is an eventual image of $c$, and
hence an eventual image of $M$, a contradiction with Lemma \ref{l:str}. We
see that if $c$ is preperiodic, then its endpoints are of the same period and
the same preperiod. Notice that by the above $c$ is non-periodic. Since
$\si_3(M)=\si_3(c)$, the endpoints of $M$ are either both periodic or both
preperiodic with the same period and preperiod.

Assume now that $c$ has non-preperiodic endpoints. We claim that $c$ cannot
be an edge of a gap $G$. Indeed, otherwise, by Theorem \ref{nwt-thm}, the gap $G$ must
at some moment collapse to a leaf. At this moment the image $\si_3^k(G)$ of
$G$ must be a collapsing quadrilateral, which means that, again, $\si_3^k(G)$ is the
convex hull $Q$ of, say, $M$ and $M'$. However, $\si_3^k(c)=\si_3^{k}(M)$ is
an edge of $Q$. This implies that $M$ is periodic and $c$ is preperiodic, a
contradiction with the assumption.

Finally, suppose that $c=\ol{xy}$ is the limit of leaves with endpoint $x$.
Together with $c$ they form an \emph{infinite cone} of leaves.
By Lemma 4.7 of \cite{bopt20}, this implies that $c$ is preperiodic, again a contradiction.
\end{proof}

\subsection{Pullback laminations}

We describe the set of symmetric laminations in terms of their comajors by
giving a criterion for a symmetric pair to be a comajor pair. Also, we
construct a specific symmetric \emph{pullback} lamination for any symmetric
pair satisfying that criterion.

\begin{definition}[Legal pairs]\label{d:legal} Suppose that a symmetric pair
$\{c,-c\}$ is either degenerate or satisfies the following conditions:

\begin{enumerate}
    \item[(a)] no two iterated forward images of $\pm c$ cross, and 

    \item[(b)] no forward image of $c$ crosses the interior of $\sh(M_c)$.
\end{enumerate}

Then $\{c, -c\}$ is said to be a \emph{legal pair}.

\end{definition}

We need a concept of a \emph{pullback} which dates back to Thurston
\cite{thu85}. Observe that even in the absence of a lamination we can extend
$\si_3$ onto given chords inside $\cdisk$, and, as long as the chords are
unlinked, this is consistent (we keep the notation $\si_3$ for such an
extension). Also, even without a lamination we call two-dimensional convex
hulls of closed subsets of $\uc$ \emph{gaps}.

\begin{definition}\label{d:pullback}
Suppose that a family $\A$ of chords is given and $\ell$ is a chord. A
\emph{pullback chord of $\ell$ generated by $\A$} is a chord $\ell'$ with
$\si_3(\ell')=\ell$ such that $\ell'$ does not cross chords from $\A$. An
\emph{iterated pullback chord of $\ell$ generated by $\A$} is a pullback
chord of an (iterated) pullback chord of $\ell$.
\end{definition}

Depending on $\A$, (iterated) pullback chords of certain chords may or may
not exist. In some cases though, several (iterated) pullback chords can be
found. While the construction beow can be given in general, we will from now
on restrict our attention to the cubic symmetric case. Lemma \ref{l:16}
follows from Lemma \ref{l:diameter} and is left to the reader.

\begin{Lemma}\label{l:16}
The only two symmetric laminations $\lam_1$, $\lam_2$ with comajors of length $\frac16$
have two critical Fatou gaps 
and are as follows.

\noindent $(1)$ The lamination $\lam_1$ has the comajor pair $(\ol{\frac16
\frac13}$, $\ol{\frac23 \frac56})$. The gap $U_1'$ is invariant; $U'_1\cap
\uc$ consists of all $\ga\in \uc$ such that $\si_3^n(\ga)\in [0, \frac12]$.
The gap $U''_1$ is invariant; $U''_1\cap \uc$ consists of all $\ga\in \uc$
such that $\si_3^n(\ga)\in [\frac12, 0]$. The gaps $U_1', U_1''$ share an
edge $\ol{0 \frac12}$; their edges are the appropriate pullbacks of
$\ol{0\frac12}$ that never separate in $\disk$ any two leaves from the
collection $\{\ol{\frac16 \frac13}$, $\ol{\frac23 \frac56}$,
$\ol{0\frac12}\}$.

\noindent $(2)$ The lamination $\lam_2$ has the comajor pair
$(\ol{\frac{11}{12}\frac{1}{12}}$, $\ol{\frac{5}{12} \frac{7}{12}})$. The
gaps $U'_2$, $U''_2$ form a period 2 cycle, and the set $(U_2'\cup U_2'')\cap
\uc$ consists of all $\ga\in \uc$ such that $\si_3^n(\ga)\in [\frac1{12},
\frac5{12}]\cup [\frac7{12}, \frac{11}{12}]$. The gaps $U_2', U_2''$ share an
edge $\ol{\frac14 \frac34}$; their edges are iterated pullbacks of
$\ol{\frac14 \frac34}$ that neither eventually cross nor eventually separate
any two leaves from the collection $\{\ol{\frac{11}{12}\frac{1}{12}}$,
$\ol{\frac{5}{12} \frac{7}{12}}$, $\ol{\frac14 \frac34}\}$.
\end{Lemma}

Though the laminations from Lemma \ref{l:16} are not pullback laminations as
described below, knowing them allows us to consider only legal pairs with
comajors of length less than $\frac16$ and streamline the proofs.

\begin{center} \textbf{Construction of a symmetric pullback lamination $\lam(c)$ for a legal pair $\{c, -c\}$.} \end{center}

\noindent\textbf{Degenerate case.} For $c\in\uc$, let $\pm \ell=\pm M_c$.
 (call $\ell$, $-\ell$ and their pullbacks
``leaves'' even though we apply this term to existing laminations, and we are
only constructing one). Consider two cases.

(a) If $\ell$ and $-\ell$ do not have periodic endpoints, then the family of
all iterated pullbacks of $\ell, -\ell$ generated by $\ell, -\ell$ is denoted
by $\mathcal{C}_c$.

(b) Suppose that $\ell$ and $-\ell$ have periodic endpoints $p$ and $-p$ of
period $n$. Then there are two similar cases. First, the orbits of $p$ and
$-p$ may be distinct (and hence disjoint). Then iterated pullbacks of $\ell$
generated by
$\ell$, $-\ell$ are well-defined (unique) until 
 the $n$-th step, when there are two iterated
pullbacks of $\ell$ that have a common endpoint $x$ and 
 share other endpoints with $\ell$. 
Two other iterated pullbacks of $\ell$ located on the other side of $\ell$
 have a common endpoint $y\ne 0$ and 
 share other endpoints with $\ell$.
These four iterated pullbacks of $\ell$ form a collapsing quadrilateral $Q$ with diagonal $\ell$;
 moreover, $\si_3(x)=\si_3(y)$ and $\si_3^n(x)=\si_3^n(y)=z$ is
 the non-periodic endpoint of $\ell$. Evidently, $\si_3(Q)=\ol{\si_3(p)\si_3(x)}$ is
the $(n-1)$-st iterated pullback of $\ell$. Then in the pullback lamination
$\lam(c)$ that we are defining we postulate the choice of \emph{only the
short pullbacks} among the above listed iterated pullbacks of $\ell$. So,
only two short edges of $Q$ are included in the set of pullbacks
$\mathcal{C}_c$. A similar situation holds for $-\ell$ and its iterated
pullbacks.

In general, the choice of pullbacks of the already constructed leaf $\hell$ is ambiguous
only if $\hell$ has an endpoint $\si_3(\pm \ell)$. In this case we
\emph{always} choose a short pullback of $\hell$. Evidently, this defines a
set $\mathcal{C}_c$ of chords in a unique way.

We claim that $\mathcal{C}_c$ is an invariant prelamination. To show that
$\mathcal{C}_c$ is a prelamination we need to show that its leaves do not
cross. Suppose otherwise and choose the minimal $n$ such that $\hell$ and
$\tell$ are pullbacks of $\ell$ or $-\ell$ under at most the $n$-th iterate
of $\si_3$ that cross. By construction, $\hell, \tell$ are not critical.
Hence their images $\si_3(\hell),$ $\si_3(\tell)$ are not degenerate and do
not cross. It is only possible if $\hell, \tell$ come out of the endpoints of
a critical leaf of $\lam$. We may assume that $\|\hell\|\ge \frac16$ (if
$\hell$ and $\tell$ are shorter than $\frac16$ then they cannot cross).
However by construction this is impossible. Hence $\mathcal{C}_c$ is a
prelamination. The claim that $\mathcal{C}_c$ is invariant is
straightforward; its verification is left to the reader. By Theorem
\ref{t:prelam-cl}, the closure of $\mathcal{C}_c$ is an invariant lamination
denoted $\lam(c)$. Moreover, by construction $\mathcal{C}_c$ is symmetric
(this can be easily proven using induction on the number of steps in the
process of pulling back $\ell$ and $-\ell$). Hence $\lam(c)$ is a symmetric
invariant lamination.

\noindent\textbf{Non-degenerate case.} As in the degenerate case, we will
talk about leaves even though we are still constructing a lamination. By
Lemma \ref{l:16}, we may assume that $|c|<\frac16$. Set $\pm M=\pm M_c, \pm
Q=\pm Q_c$. If $d$ is an iterated forward image of $c$ or $-c$, then, by
Definition \ref{d:legal}(b), it cannot intersect the interior $Q$ or $-Q$.
Consider the set of leaves $\mathcal{D}$ formed by the edges of $\pm Q$ and
$\cup \bigcup_{m=0}^{\infty}\{\sigma_3^{m}(c),\sigma_3^{m}(-c)\}$. It follows
that leaves of $\mathcal{D}$ do not cross among themselves. The idea is to
construct pullbacks of leaves of $\mathcal{D}$ in a step-by-step fashion and
show that this results in an invariant prelamination $\mathcal{C}_c$ as in
the degenerate case.

More precisely, we proceed by induction. Set $\mathcal{D}=\mathcal{C}_c^0$.
Construct sets of leaves $\mathcal{C}_c^{n+1}$ by collecting pullbacks of
leaves of $\mathcal{C}_c^{n}$ generated by $Q$ and $-Q$ (the step of
induction is based upon Definition \ref{d:legal} and Definition
\ref{d:pullback}). The claim is that except for the property (D2)(1) from
Definition \ref{inv-lam} (a part of what it means for a lamination to be
backward invariant), the set $\mathcal{C}_c^n$ has all the properties of
invariant laminations listed in Definition \ref{inv-lam}. Let us verify this
property for $\mathcal{C}_c^1$. Let $\ell\in \mathcal{C}_c^1$. Then
$\si_3(\ell)\in \mathcal{D}$, so property (D1) from Definition \ref{inv-lam}
is satisfied. Property (D2)(2) is, evidently, satisfied for edges of $Q$ and
$-Q$. If $\ell$ is not an edge of $\pm Q$, then, since leaves $\pm
\si_3(Q)=\si_3(\pm c)$ do not cross $\si(\ell)$, and since on the closure of
each component of $\uc\sm [Q\cup -Q]$ the map is one-to-one, then $\ell$ will
have two sibling leaves in $\mathcal{C}_c^1$ as desired. Literally the same
argument works for $\ell\in \mathcal{C}_c^{n+1}$ and proves that each set
$\mathcal{C}_c^{n+1}$ has properties (D1) and (D2)(2) from Definition
\ref{inv-lam}. This implies that $\bigcup_{i\ge 0}
\mathcal{C}_c^i=\mathcal{C}_c$ has all properties from Definition
\ref{inv-lam} and is, therefore, an invariant prelamination. By Theorem
\ref{t:prelam-cl}, its closure $\lam(c)$ is an invariant lamination.

The lamination $\lam(c)$ is called the \emph{pullback lamination (of $c$);}
we often use $c$ as the argument, instead of the less discriminatory $\{c,
-c\}$.

\begin{Lemma}\label{pull-back-lam1}
A legal pair $\{c,-c\}$ is the comajor pair of the symmetric lamination
$\lam(c)$. A symmetric pair $\{c, -c\}$ is a comajor pair if and only if it
is legal.
\end{Lemma}

\begin{proof} The verification of the fact that $\{c, -c\}$ is the comajor pair
of $\lam(c)$ is straightforward; we leave it to the reader. On the other
hand, a comajor pair of a symmetric lamination is legal by Lemma
\ref{short-leaves}.
\end{proof}

\subsection{The lamination of comajors}

\begin{definition}\label{under-dfn}
For a non-diameter chord $n=\overline{ab}$, the smaller of the two arcs into
which $n$ divides $\uc$ is denoted by $H(n)$. Denote the closed subset of
$\overline{\D}$ bounded by $n$ and $H(n)$ by $R(n)$. Given two comajors $m$
and $n$, write $m\prec n$ if $m\subset R(n)$,
 and say that $m$ is \emph{under} $n$.
\end{definition}

Note that, if $m\prec n$, then any set of pairwise non-crossing chords that
separate $m$ from $n$ in $\disk$ is linearly ordered by $\prec$.

\begin{Lemma}\label{cmajorL0}
Let $\{c,-c\}$ and $\{d,-d\}$ be legal pairs, where $c$ is  degenerate and
$c\prec d$. Suppose that either $c$ is not an endpoint of $d$, or $\si_3(c)$
is not periodic. Then the leaves $\si_3^n(d)$ with $n\ge 1$ are disjoint from
the majors of $\lam(c)$. In particular, if the endpoints of $\si_3(d)$ are
non-periodic then the leaves $\si_3^n(d)$ with $n\ge 1$ are disjoint from
$\pm M_c$.
\end{Lemma}

\begin{proof}
Let the majors of $\lam(c)$ be critical leaves $M$ and $-M$; let the majors
of $\lam(d)$ be leaves $N$, $N'$, $-N$, $-N'$. Clearly, $M$ and $-M$ lie
(except, perhaps, for the endpoints) in $\sh(N)$ and separate (except,
perhaps, for the endpoints) $N$ from $N'$ and $-N$ from $-N'$. The claim
holds by Lemma \ref{l:str} if $c$ is not an endpoint of $d$. If $c$ is an
endpoint of $d$, then by the assumption $\si_3(c)$ is nonperiodic. Thus, the
endpoints of $d$ and those of the majors $N$ and $N'$  are nonperiodic by
Lemma~\ref{cmajor-end-points}. Note that $\si_3(d)=\si_3(N)=\si_3(N')$. If
our claim fails, then $\si_3^n(d)=\si_3^n(N)=\si_3^n(N')$ shares an endpoint
with (1) the majors $M$ and $N$, or (2) the majors $-M$ and $-N$.
In both cases, the notation for the majors is chosen so that $M$ and $N$ (then also $-M$ and $-N$)
 have a common endpoint.
Thus, (1) means $\si_3^n(d)\cap M\cap N\ne\varnothing$, and (2) means
$\si_3^n(d)\cap (-M)\cap (-N)\ne\varnothing$. Consider these two cases.

(1) Let $\si_3^n(d)=\si_3^n(N)$ share an endpoint $y$ with $M$ and $N=\ol{xy}$.
Observe that, by Lemma \ref{l:str}, the leaf $N$ never maps to
its short strips. Applying $\si_3^n$ to $N\cup \si_3^n(N)$ we see that
$\si_3^{2n}(N)$ is concatenated to $\si_3^n(N)$ and the vertices of leaves
$N,$ $\si_3^n(N)$, and $\si_3^{2n}(N)$ are ordered positively or negatively
on $\uc$. If we continue, we will see that further $\si_3^n$-images of $N$
are ordered in the same fashion. This implies that at some moment this chain
of leaves will connect to the endpoint $x$ of $N$ (recall that $\si_3^n$ is a
local expansion), and $N$ will turn out to be periodic, a contradiction.

(2) If $\si_3^n(d)=\si_3^n(N)$ shares an endpoint with $-M$ and $-N/-N'$
(say, $-N$), then, by symmetry, $\si_3^n(-d)=\si_3^n(-N)$ shares an endpoint
with $M$ and $N$. Thus, leaves $N,$ $\si_3^n(-N),$ and $\si_3^{2n}(N)$ are
concatenated. The idea, as before, is to apply the appropriate iterate of
$\si_3$ (in this case $\si_3^{2n}$) that shifts $N$ to the next occurrence of
this leaf in the concatenation and use the fact that any concatenation like
that is one-to-one and orientation preserving. There are two cases here.

(2a) Suppose that $N=\ol{xy}$ and $\si_3^n(-N)=\ol{yz}$ are oriented in one
way while $\si_3^n(-N)=\ol{yz}$ and $\si_3^{2n}(N)=\ol{zu}$ are oriented
differently. E.g., suppose that $x>y>u>z$ (so that the triple $x, y, z$ is
negatively oriented while the triple $y, z, u$ is positively oriented). Then,
if $\si_3^{2n}(\si_3^n(-N))=\ol{uv}$, then $z, u, v$ must also be negatively
oriented and so all these points are ordered on the circle as follows:
$x>y>u>v>z$. Repeating this over and over we will see that leaves $N,$
$\si_3^{2n}(N),$ $\dots,$ $\si_3^{k\cdot 2n}(N)$ are consecutively located
under one another. However, this is impossible as $\si_3$ is a local
expansion.

(2b) Suppose that $N=\ol{xy}$ and $\si_3^n(-N)=\ol{yz}$ are oriented in the
same way as $\si_3^{2n}(N)=\ol{zu}$. Iterating $\si_3^{2n}$ on these two
leaves we see, similar to (a), that all the images of $x, y, z$ are oriented
in the same way as $x, y, z$ themselves. Hence, again, the points
$\si_3^{k\cdot 2n}(x)$ form a sequence of points that converges back to $x$
which is impossible unless on a finite step the process stops because the
next link in the concatenation dead-ends into the point $x$. The leaf from
the concatenation with an endpoint $x$ is an image of $N=\ol{xy}$ or
$\si_3^n(-N)=\ol{yz}$. Suppose that it is an image of $N$. Then, the next
image of $\si_3^n(-N)=\ol{yz}$ is forced to coincide with $N$ because it
cannot enter short strips of $N$. The thus constructed finite polygon maps by
$\si_3^{2n}$ onto itself and has all edges periodic, a contradiction with $N$
being non-periodic. If the leaf from the concatenation with endpoint $x$ is
some image of $\si_3^n(-N)=\ol{yz}$, then it immediately follows that $N$ is
periodic, again a contradiction.

Thus, the leaves $\si_3^n(d), n\ge 1$ are disjoint from $\pm M$ as claimed.
\end{proof}

\begin{Lemma}\label{cmajor-comp-d}
Let $\{c,-c\}$ and $\{d,-d\}$ be legal pairs, where $c$ is degenerate and
$c\prec d$. Suppose that $c$ is not an endpoint of $d$, or $\si_3(c)$ is not
periodic. Then $d\in \lam(c)$. In addition, the following holds.

\begin{enumerate}

\item Majors $D, D'$ of $\lam(d)$ are leaves of $\lam(c)$ unless $\lam(c)$
    has two finite gaps $G, G'$ that contain $D, D'$ as their diagonals,
    share a critical leaf $M_c=M$ of $\lam(c)$ as a common edge, and are
    such that $\si_3(G)=\si_3(G')$ is a preperiodic gap.

\item If majors of $\lam(d)$ are leaves of $\lam(c)$ and $\ell\in\lam(d)$ is
    a leaf that never maps to a short side of a collapsing quadrilateral of
    $\lam(d)$, then $\ell\in\lam(c)$.

\end{enumerate}

\end{Lemma}


\begin{proof}
We claim that the iterated images of $d$ do not intersect leaves of
$\lam(c)$. By Lemma \ref{cmajorL0}, no iterated image of $d$ intersects the
majors $\pm M_c=\pm M$ of $\lam(c)$. Let an iterated image $\ell_d$ of $d$
intersect an iterated pullback $\ell_M$ of $M$ or $-M$. If they share an
endpoint, then after a few steps we will arrive at an iterated image of $d$
that shares an endpoint with $M$ or $-M$, a contradiction. Suppose that
$\ell_{d}$ crosses $\ell_M$. The only way $\si_3(\ell_d)$ and $\si_3(\ell_M)$
``lose'' their crossing is when $\ell_d$, $\ell_M$ ``come out'' of the
distinct endpoints of a critical leaf. Since, by Lemma \ref{cmajorL0}, the
leaf $\ell_d$ is disjoint from $\pm M$, this is impossible. Hence
$\si_3(\ell_d)$ and $\si_3(\ell_M)$ cross. Repeating this argument, we see
that the associated iterated images of $\ell_d$ and $\ell_M$ cross each
other. Since $\ell_M$ is mapped to $M$ or $-M$ under a finite iteration of
$\si_3$, in the end we will have an image of $d$ crossing $M$ or $-M$, a
contradiction.

So $d$ is a leaf of $\lam(c)$ or a diagonal of a gap in $\lam(c)$. Let us
rule out the latter. Since $\lam(c)$ has two critical leaves, there are no
gaps of $\lam(c)$ on which $\si_3$ has degree $m>1$; suppose, by way of
contradiction, that $d$ is a diagonal of a gap $G$ of $\lam(c)$, and consider
cases.

(a) If no iterated image of $G$ has a critical edge, then by Theorem
\ref{nwt-thm}, the gap $\si_3^k(G)$ is periodic for some minimal $k\ge 0$,
and, by Theorem \ref{t:gaps-1}, the gap $\si_3^k(G)$ is finite. A
contradiction with Corollary \ref{no-diagonal}.

(b) Suppose that the gap $\si_3^m(G)$ has a critical edge for a minimal $m\ge
0$. Consider two cases. First, suppose that $c$ is strictly under $d$. Since
$G$ is a gap of $\lam(c)$ containing $d$ as a diagonal, then there are two
cases. First, there may exist two sibling gaps  of $G$ separated in $\disk$
by the critical leaf $M$ of $\lam(c)$, but themselves non-critical. Each such
gap contains a major $M_d$ or $M'_d$ as a diagonal. However, $\si_3^m(G)$ has
a critical edge which then implies that $d$ is mapped into its own short
strips, a contradiction with $d$ being legal. Now, the second case is when
there are two gaps of $\lam(c)$, denoted by $A$ and $A'$, that share $M$ as a
common edge and contain $M_d$ and $M'_d$, respectively. Evidently,
$\si_3(G)=\si_3(A)=\si_3(A')$. Since the gap $\si_3^m(G)$ has a critical
edge, we may assume that $\si_3^m(G)=A$. It follows that $G$ cannot be
finite.

Since $\si_3^m$ is one-to-one on the vertices of $G$, we have that
$\si_3^m(d)$ is a diagonal of $\si_3^m(G)=H$. Since $G$ is infinite, $H$ is
(pre)periodic (by Theorem \ref{t:kiwi}). Since by Theorem \ref{t:gaps-1} the
cycle of gaps from the orbit of $H$ must have at least one gap with critical
edge, then $H$ itself is periodic. Since images of $d$ do not cross each
other, $H$ is not a Siegel gap. Hence $H$ is a caterpillar gap. Since by
Lemma \ref{cmajorL0}, the iterated images of $d$ are disjoint from $\pm M$,
then by Lemma \ref{l:adiag-1}, an eventual image of $d$ is a periodic
diagonal of $H$. We claim that this is impossible.

We may assume that $M$ is an edge of $H$. By Theorem \ref{t:gaps-1}, an
endpoint of $M$ is periodic. Then by the assumptions $c$ is not an endpoint
of $d$, and by Lemma \ref{l:str} the orbit of $d$ is disjoint from that of
$M$. Hence $\partial H$ contains two cycles, that of $\si_3(M)=\si_3(c)$, and
that of an endpoint of a periodic image of $d$.
Since the images of $d$ are
diagonals, $\partial H$ contains 3 periodic points from 2 cycles. This
allows one to connect a certain triple of points from these two cycles so that they
form a fixed return triangle $T$.
Consider the forward orbit of $T$ and then
 the grand orbit of $T$, where iterated pullbacks of $T$ and  of its iterated images are constructed
 consistently with $\lam(c)$.
Since by our assumption
$\lam(c)$ has caterpillar gaps with edges $\pm M$, it is easy to see that
this yields a cubic symmetric lamination with a fixed return triangle, a
contradiction with Lemma \ref{P-or-P}. So, $d$ is an edge of $G$ and a leaf
of $\lam(c)$, and so is $-d$. Let us now prove the remaining claims.

(1) Consider the critical quadrilateral $Q$ of $\lam(d)$ with
 $\si_3(Q)=\si_3(d)$.
Two long/medium edges of $Q$ are majors $D$, $D'$ of
$\lam(d)$. If $c$ and $d$ are disjoint, then the remaining two short edges of
$Q$ cross $M$ and cannot be leaves of $\lam(c)$. Hence in that case $D$, $D'$
are leaves of $\lam(c)$ as desired. Consider the case when $c$ is an endpoint
of $d$. Then $M$ is a (critical) diagonal of $Q$, and both endpoints of $M$
are non-periodic (this is because by our assumptions $\si_3(c)=\si_3(M)$ is
non-periodic). Suppose that $D$, $D'$ are not leaves of $\lam(c)$. By
properties of laminations two edges of $Q$ (say, $q$ and $q'$) are leaves of
$\lam(c)$. By our assumptions there are gaps $G, G'$ that contain $D$, $D'$ as
their diagonals and share a critical leaf $M$ of $\lam(c)$ as a common edge.

We claim that $G, G'$ are finite. Indeed, if they are infinite, then they are
(pre)periodic. Since $\lam(c)$ is cubic and has two critical leaves, the
cycle of infinite gaps to which $G$ and $G'$ eventually map has a gap with a
critical edge. It follows that one of the gaps $G, G'$ (say, $G$) is
periodic, and the first return map to $G$ is of degree one. By Theorem
\ref{t:gaps-1}, consider caterpillar and Siegel cases. Suppose that $G$ is
caterpillar.
Then, by Theorem \ref{t:gaps-1}, the leaf $\si_3(M)=\si_3(c)$ is periodic,
a contradiction with the assumptions. Suppose that $G$ is Siegel. Then by
Theorem \ref{t:gaps-1}, both $q$ and $q'$ must eventually map to $M$ which implies
that an endpoint of $M$ is periodic, again a contradiction. Thus, $G$ and
$G'$ are finite.
Since $D$, $D'$ are diagonals of $G$, $G'$, respectively, $\si_3(G)$ is a
gap (not a leaf). By Theorem \ref{nwt-thm} and by the assumptions $\si_3(G)$
is preperiodic.

(2) Observe that by the assumptions $d$, $D$, $D'$ and their iterated images
all belong to $\lam(c)$. Denote this family of leaves by $X$. We claim that
iterated pullbacks of these leaves are leaves of $\lam(c)$. First consider a
leaf $\ell\in \lam(d)$ such that $\si_3^n(\ell)=x\in X$. We claim that
$\ell\in \lam(c)$. Let us use induction over $n$. The base of induction is
already established as $X\subset \lam(c)$. Suppose that the claim is proven
for $n=k$ and prove it for $n=k+1$. Consider $\ell\in \lam(d)$ such that
$\si_3^{k+1}(\ell)=x\in X$. Then, by induction, $\si_3(\ell)\in \lam(c)$.
Now, by properties of laminations, this implies that $\ell\in \lam(c)$, too,
unless, say, the following holds: $\ell$ shares an endpoint with $M$, there
is another chord $t$ that forms a triangle with $\ell$ and $M$, and in fact
$t$ is a leaf of $\lam(c)$ while $\ell$ is not (other cases are similar). We
claim that this is impossible. Indeed, if $\ell$ shares an endpoint with $M$
and is disjoint from the interior of $Q$, then $t$ must cross $D$, a
contradiction as $D$ is a leaf of $\lam(c)$ by the assumptions, and cannot be
crossed by another leaf of $\lam(c)$. Hence, $\ell\in \lam(c)$, as desired.

Consider the iterated pullbacks of leaves of $X$ that are leaves of
$\lam(d)$. By the previous paragraph they are leaves of $\lam(c)$. Hence the
closure of this set of leaves is also a subset of $\lam(c)$. Therefore, the
only possible leaves of $\lam(d)$ that are not leaves of $\lam(c)$ are
iterated pullbacks of the short edges of $\pm Q$ and their limits). However,
in the statement of the lemma we explicitly exclude leaves $\ell$ that are
iterated pullbacks of short edges of $\pm Q$. Hence it suffices to show that
the lengths of these pullbacks converges to zero (this will imply that limits
of pullbacks of the short edges of $\pm Q$ are points of $\uc$). This follows
from Lemma~\ref{col-quad}.
\end{proof}

Let us prove an important property of pullback laminations.

\begin{Lemma}\label{l:major-pb}
Let $d$ be a comajor. Then the iterated pullbacks of the majors $\pm M_d$ and
$\pm M'_d$ of $\lam(d)$ are dense in the pullback lamination $\lam(d)$ with,
possibly, one exception: the leaves of $\lam(d)$ that are short sides of
critical quadrilaterals of $\lam(d)$ and their iterated pullbacks might not
be approximated by iterated pullbacks of the majors of $\lam(d)$. Thus,
iterated pullbacks of the minors of $\lam(d)$ are dense in $\lam(d)$.
\end{Lemma}

\begin{proof}
If $d$ is degenerate, the claim follows from the definitions. Let $d$ be
non-degenerate. Then there are two cases. First, assume that $M=M_d$ has a
periodic endpoint. Then by Lemma \ref{cmajor-end-points}, the leaf $M$ is periodic, and
so is $-M$.
It follows that the iterated pullbacks of the majors of $\lam(d)$ form the same set as
 the iterated pullbacks of the majors, comajors and \emph{all their iterated forward images}
 used in the construction of the pullback lamination $\lam(d)$.
Hence, in this case, the claim follows from the definitions.

From now on assume that the endpoints of the majors $\pm M$, $\pm M'$ are
non-periodic. By Lemma \ref{cmajor-end-points}, the leaf $d$ is not periodic either,
and, moreover, no endpoint of $d$ is periodic. It follows that the minors $\pm
\si_3(M)$ have no periodic endpoints. Choose an endpoint $c$ of $d$ and
consider $\lam(c)$; the critical sets of $\lam(c)$ are leaves $\pm M_c=\pm
\ell$. Evidently, Lemma \ref{cmajor-comp-d} applies to $\lam(d)$ and
$\lam(c)$. 

Let $y\in \lam(d)\cap \lam(c)$.
Since $y$ is not eventually mapped to $\ell$ (as $y\in \lam(d)$),
 then $y$ is approximated by iterated pullbacks of $\pm \ell$. Set
$\pm Q=\pm Q_d$. By definition, pullbacks of $\pm \ell$ that converge to $y$
are diagonals of the pullbacks of $\pm Q$ corresponding to them. Denote by
$N$ a short side of $\pm Q$. The leaves $\pm \si_3(d)$ have 5 preimage-leaves
(the edges of $\pm Q$ and $\pm d$) while all other leaves have 3
preimage-leaves. In particular, every leaf that is shorter than $\si_3(d)$
has three even shorter preimages. This applies to $N$, and the length of the
$n$-th pullback of $N$ is $\frac{||N||}{3^n}$. Hence, $y$ is a limit of
iterated pullbacks of the majors of $\lam(d)$. Since, by Lemma
\ref{cmajor-comp-d}, all chords $y=\pm \si_3^n(d)$, where $n\ge 0$, are
leaves of $\lam(c)$, they all are limits of iterated pullbacks of the majors
of $\lam(d)$.

Now, let $y\in \lam(d)\sm \lam(c)$.
We may assume that $y$ is not eventually mapped to
 an edge of $\pm Q$. Then $y$ is the limit of iterated pullbacks of $\pm Q$, or, if
not, the limit of iterated pullbacks of leaves $\pm \si_3^n(d)$. In the
former case, the argument from the previous paragraph applies.
In the latter case, by the previous paragraph, the fact that leaves $\pm
\si_3^n(d)$ are limits of iterated pullbacks of the majors of $\lam(d)$
implies that iterated pullbacks of leaves $\pm \si_3^n(d)$ avoiding $\pm Q$
are also limits of iterated pullbacks of the majors of $\lam(d)$. Thus,
iterated pullbacks of the majors of $\lam(d)$ are dense among all leaves of
$\lam(d)$, except, possibly, for the leaves of $\lam(d)$ that are pullbacks
of the short sides of critical quadrilaterals $\pm Q$ of $\lam(d)$.
\end{proof}

\begin{theorem}\label{cmajor-thm}
Distinct comajors of symmetric laminations do not cross.
\end{theorem}

\begin{proof}
Let $\{c_1,-c_1\},$ $\{c_2,-c_2\}$ be pairs of comajors of symmetric
laminations $\lam_1$ and $\lam_2$, respectively. If $c_1$ crosses $c_2$, then
$H(c_1)\cap H(c_2)\neq\emptyset$. Choose a non-preperiodic point $p\in
H(c_1)\cap H(c_2)$. The symmetric lamination $\lam(p)$ has comajors
$\{p,-p\}$; since $p\prec c_1$ and $p\prec c_2$, then by Lemma
\ref{cmajor-comp-d} both $c_1$ and $c_2$ are leaves of $\lam(p)$, a
contradiction.
\end{proof}

The next result follows from Theorem \ref{cmajor-thm} and Theorem 2.15.
\begin{theorem} The space of all symmetric laminations is compact. The
set of all their non-degenerate comajors is a lamination.
\end{theorem}

Definition \ref{d:cscl} is an analogue of Thurston's definition of QML.

\begin{definition}\label{d:cscl}
The set of all chords in $\D$ which are comajors of some symmetric lamination
is a lamination called the \emph{Cubic Symmetric Comajor Lamination}, denoted
by $C_sCL$.
\end{definition}

Note that $C_sCL$ satisfies symmetric property (D3) as all comajors come in symmetric pairs.

\section{Cubic Symmetric Comajor Lamination is a q-lamination}

By Corollary \ref{cmajor-end-points}, all non-degenerate comajors are
non-periodic. We classify them as \textit{preperiodic of preperiod 1,
preperiodic of preperiod bigger than 1,} and \emph{not eventually periodic},
and consider each case separately. By Lemma \ref{cmajor-end-points} a comajor
of preperiod 1 and period $k$ corresponds to a periodic major and maps to the
major by $\si_3^k$.

\begin{Lemma}\label{pre-period-1}
A comajor leaf of preperiod 1 is disjoint from all other comajors in $C_sCL$.
\end{Lemma}

\begin{proof}
By Theorem \ref{cmajor-thm}, intersecting comajors share an endpoint. Then,
by Lemma \ref{cmajor-end-points}, they have the same preperiod and period.
Thus, a comajor of preperiod 1 can only share an endpoint with a comajor of
the same kind. Assume that there exist distinct compajor pairs $\{c,-c\},$
$\{d,-d\}$ of preperiod 1 and period $k$ such that $c$ and $d$ share an
endpoint $a$. Since $\si_3(c)$ is a periodic leaf, there is a periodic leaf
that maps to $\si_3(c)$. By Lemma \ref{cmajor-end-points}, this periodic leaf
is a major of $\lam(c)$.

We claim that $c$ is under $d$ or $d$ is under $c$. Indeed, otherwise
$c=\ol{xa}$ and $d=\ol{ay}$ are located next to each other. Let $x<a<y$ and,
hence, $\si_3(x)<\si_3(a)<\si_3(y)$. Consider the periodic majors $M$ of
$\lam(c)$ and $N$ of $\lam(d)$. Evidently, they share an endpoint $A$ (with
$\si_3(A)=\si_3(a)$) and have other endpoints $X$ (with $\si_3(X)=\si_3(x)$) and
$Y$ (with $\si_3(Y)=\si_3(y)$) so that $M=\ol{AX}$ and $N=\ol{AY}$. Since majors
are long/medium leaves, it is easy to see that $X>A>Y$ (the orientation
changes). We claim that this is impossible. Indeed, the short strip $C(M)$
and the short strip $C(N)$ have a common diagonal $\ol{AZ}$ where  $Z$ is the
remaining sibling point of $a$ and $A$.
Since the iterated images of $c$ do not enter the
interior of $\sh(M)$ and images of $d$ do not enter $\sh(N)$,  for any
$j$, the convex hull of points $\si_3^j(x), \si_3^j(a), \si_3^j(y)$ is
disjoint from critical chords $\ol{AZ}, -\ol{AZ}$. Hence, the orientation of
this triple of points must not change, which contradicts the fact that $\si_3^k(x)=X>
\si_3^k(a)=A>\si_3^k(y)=Y$ while $x<a<y$.
This contradiction shows that we may assume that $c$ is
located under $d$.

Let $x$ be a non-preperiodic point under $c$.
By Lemma \ref{cmajor-comp-d}, the chords
$M$ and $N$ are leaves of $\lam(x)$. Hence $\lam(x)$ has a gap $G$ such that
$M$ and $N$ are edges of $G$. Since $M$ and $N$ are periodic, $G$ is periodic
too. By Proposition \ref{P-or-P}, the first return map on $G$ is not a fixed
return map. This implies that $N$ enters the interior of $\sh(N)$, a
contradiction.
\end{proof}

A leaf of a lamination is a \emph{two sided limit leaf} if it is not on the
boundary of a gap, i.e., if it is a limit of other leaves from both sides
(e.g., by Lemma~\ref{cmajor-end-points}, all non-preperiodic comajors are two
sided limit leaves). A lamination can have periodic or preperiodic two sided
limit leaves. We prove that a two sided limit comajor $c$ of $\lam(c)$ is a
two sided limit leaf in the {Cubic
Symmetric Comajor Lamination} $C_sCL$, too. 

\begin{Lemma}\label{appr-under}
Let $c\in C_sCL$ be a non-degenerate comajor. If $\ell\in \lam(c),$
$\ell\prec c$ and $\|\ell\|>\frac{\|c\|}{3}$, then $\ell\in C_sCL$. In
particular, if $c_i\in \lam(c),$ $c_i\prec c$ and $c_i \rightarrow c$, then
$c_n\in C_sCL$ for sufficiently large $n$.
\end{Lemma}

\begin{proof}
Choose $\ell\in \lam(c)$ with $\|\ell\|>\frac{\|c\|}{3}$.
We claim that the
leaves $\{\ell,-\ell\}$ form a legal pair (see Definition \ref{d:legal}).

(a) Clearly, no forward images of $\ell$ and $-\ell$ cross.

(b) Let $m$ be a minor of $\lam(c)$. Since $\|\ell\|<\|c\|$, then
$\|\sigma_3(\ell)\|=3\|\ell\|<\|\si_3(c)\|=\|m\|$. By Lemma
\ref{short-leaves}, no forward image of $\sigma_3(\ell)$ is shorter than
$3\|\ell\|$. 

(c) The long and medium sibling chords $M_{\ell}$ and $M_{\ell}'$ of $\ell$
are located inside the short strips $C(M)$, $C(-M)$ of a major $M=M_c$ of
$\lam(c)$. An iterated image $\tell$ of $\ell$ cannot cross majors of
$\lam(c)$. Hence $\tell$ is either outside of $\sh(M)$ or inside it. We claim
that $\tell$ is outside. Indeed, if $\tell$ is inside, say, $C(M)$, it cannot
be closer to criticality than $M$. On the other hand, $\|\tell\|\ge 3\|\ell\|
> \|c\|$. This implies that $\tell$ cannot be inside $\sh(M)=C(M)\cup C(-M)$.
Hence leaves from the forward orbit of $\ell$ do not cross chords $\pm
M_{\ell}$ and $\pm M'_{\ell}$.

Thus, $\{\ell, -\ell\}$ is a legal pair and so $\ell \in C_sCL$ as desired.
\end{proof}

Consider now comajors approximated from the other side.

\begin{Lemma}\label{simple-pull-back-lam}
Let $\lam$ be a symmetric lamination with comajors $\{c,-c\}$. Suppose there
is a short leaf $\ell_s\in \lam$ satisfying the conditions below:
\begin{enumerate}
    \item[(i)] $c\prec\ell_s$,
    \item[(ii)] the leaf $\ell_m=\sigma_3(\ell_s)$ never maps under itself
        or under $-\ell_m$. 
\end{enumerate}
Then there is a symmetric lamination $\lam(\ell_s)$ with comajors
$\{\ell_s,-\ell_s\}$.
\end{Lemma}

\begin{proof}
(a) Since $\ell_s, -\ell_s\in \lam$, all forward images of $\ell_s, -\ell_s$
do not cross.

(b) The siblings of $\ell_s$ in $\lam$ are either both short or one long and
one medium leaf. Since $\ell_s\succ c$, a short sibling of $\ell_s$ (or its
image under the rotation by $180$ degrees) would intersect the major leaves
of $\lam$. Thus, the siblings of $\ell_s$ (and their rotations by $180$
degrees) in $\lam$ are long and medium. Hence forward images of $\ell_s$ do
not cross the long and medium siblings of $\ell_s$ (or their rotations by
$180$ degrees).

(c) Assume that, for some $k>0$, we have $\|\sigma_3^k(\ell_s)\|<3\|\ell_s\|$ for the
first time. This implies that $\sigma_3^{k-1}(\ell_s)$ is closer to
criticality than the long and medium sibling leaves of $\ell_s$. Hence the
leaf $\sigma_3^k(\ell_s)=\si_3^{k-1}(\ell_m)$ is under $\ell_m$ or $-\ell_m$
contradicting the assumptions.

By definition, $\{\ell_s,-\ell_s\}$ is a legal pair, and by Lemma
\ref{pull-back-lam1}  there exists a symmetric lamination $\lam(\ell_s)$ with
$\{\ell_s,-\ell_s\}$ as a comajor pair.
\end{proof}

\begin{definition}\label{d:symmetr}
Let $\ell$ be a leaf of a symmetric lamination $\lam$ and $k>0$ be such that
$\sigma_3^k(\ell)\neq\ell$ (in particular, the leaf $\ell$ is not a
diameter). If the leaf $\sigma_3^k(\ell)$ is under $\ell$, then we say that
the leaf $\ell$ \emph{moves in} by   $\sigma_3^k$; if $\sigma_3^k(\ell)$ is
not under $\ell$, then we say that the leaf $\ell$ \emph{moves out} by
$\sigma_3^k$. If two leaves $\ell$ and $\hell$ with $\ell \prec \hell$ of the
same lamination both move in or both move out by the map $\sigma_3^k$, then
we say that the leaves
 \emph{move in the same direction}. If one  of the leaves
$\{\ell,\hell\}$ moves in and the other moves out, then we say that the
leaves \emph{move in the  opposite directions}. There are two ways of
moving in the opposite directions: if  $\ell$ moves out and $\hell$ moves in,
we say they \textit{move towards each other}; if $\ell$ moves in and $\hell$
moves out, we say that they  \textit{move away from each other}.
\end{definition}

Since $\lam$ is a symmetric lamination, then the maps $\si_3$ and $-\si_3$
both map $\lam$ onto itself. 

\begin{Lemma}\label{move-twrds}
Let $\hell\neq\ell$ be non-periodic leaves of a symmetric lamination $\lam$
with $\hell\succ\ell$. Given an integer $k>0$,  let
$h:\uc\to \uc$ be either the map $\si_3^k$ or the map $-\si_3^k$. Suppose
that the leaves $\ell$ and $\hell$ move towards each other by the map $h$ and
neither the leaves $\ell$ and $\hell$, nor any leaf separating them, can
eventually map into a leaf (including degenerate) with both endpoints in one
of the boundary arcs of the strip $\mathcal{S}(\ell,\hell)$. Then there
exists a $\si_3$-periodic leaf $y\in\lam$ that separates $\ell$ and $\hell$.
\end{Lemma}

\begin{proof}
Note that, if $h=-\si_3^k$, then $h^2=\si_3^{2k}$. Hence an $h$-periodic leaf
is $\si_3$-periodic, too. We will now show that there exists an $h$-periodic
leaf $\ell''$ separating $\ell$ and $\hell$. Consider the family $T$ of
leaves of $\lam$ that consists of $\ell$ and leaves $u$ separating $\ell$
from $\hell$ and either $h(u)=u$ or $u$ separates $\ell$ from $h(u)\sm u$.
By continuity, $T$ is closed.
Also, $T$ is nonempty as $\ell\in T$ by definition.
Hence $T$ contains a leaf $t$ farthest from $\ell$.
If $h(t)=t$ we are done; assume that $t\not=h(t)$. By
continuity and by the choice of $t$ there must exist a gap $H$ whose interior
is separated from $\ell$ in $\disk$ by $t$, and $t$ is an edge of $H$.
Let
$s$ be the edge of $H$ defined as follows: if $\hell$ is an edge of $H$, then
$s=\hell$, otherwise $s$ is the edge of $H$ that separates $\hell$ from $h(s)\sm s$.
If $h(s)=s$, we are done.
Assume that $h(s)\ne s$; then, since
$s\not\in T$ and by the assumptions, $h(s)=t,$ $h(t)=s$, and $H$ is
$h$-invariant. Hence $s$ and $t$ are $h$-periodic, and we are done in that
case, too.
\end{proof}

\begin{Lemma}\label{appr-over}
Let $c\in C_sCL$ be a non-degenerate comajor such that $\si_3(c)$ is not
periodic. If there exists a sequence of leaves $c_i\in \lam(c)$ with $c\prec
c_i$ and $c_i\rightarrow c,$ then $c$ is the limit of preperiodic comajors
$\hc_j \in \lam(c)$ of preperiod 1 with $c\prec \hc_j$ for all $j$.
\end{Lemma}

\begin{proof}
Let $\{m,-m\}$ and $\{M,-M\}$ be the minors and majors of $\lam(c)$
respectively (we choose one pair of majors out of two possible pairs). By the
assumptions, the minor leaves $m=\si_3(c)$ and $-m=-\si_3(c)$ are not
periodic. Set $m_i=\si_3(c_i)\in \lam(c)$; then $m_i\rightarrow m$ and
$m_i\succ m$.

By Lemma \ref{l:major-pb}, iterated pullbacks of minors are dense in
$\lam(c)$. Hence there exists a sequence $n_i$ of further and further
preimages of $m$ or $-m$ with $n_i\succ m$ and $n_i\rightarrow m$ (as each
$m_i$ is approximated by similar sequences of pullbacks of minors). For each
$i$ there is $k_i$ such that for $h_i=\si_3^{k_i}$ or $h_i=-\si_3^{k_i}$ we
have $h_i(n_i)=m$. Because no forward image of $m$ can be shorter than $m$,
the leaf $h_i(m)$ cannot be under $m$.
Also, 
$h_i(m)\neq m$
(recall that $h_i^2(m)=\si_3^{2k_i}(m)$, and $m$ is not $\si_3$-periodic).
Thus, $h_i$ maps $n_i$ and $m$ towards each other.

Choose $n_i$ so that the width of the strip $\mathcal{S}(m,n_i)$ is less than
$\|m\|$. By Lemma \ref{short-leaves}, any leaf of length at least $\|m\|$
never maps into the boundary arcs of the strip $\mathcal{S}(m,n_i)$. Since
$n_i$ is not a periodic leaf, by Lemma \ref{move-twrds}, there is a
$\si_3$-periodic leaf $y_i$ separating $m$ and $n_i$.

Choose the shortest leaf $\hy_i$ in the orbit of $y_i$. We claim that if
$\hy_i\ne y_i$ then it separates either $y_i$ and $m$, or $-y_i$ and $-m$.
Indeed, $\hy_i=\si_3(\ty_i)$ with $\ty_i$ being a leaf from the orbit of
$y_i$;
the leaf $\ty_i$ is closer to a major than the corresponding pullbacks of $y_i$
 as otherwise its length will not drop
below the length of $y_i$. This implies the above made claim about the
possible locations of $\hy_i$.

Choose long and medium pullbacks of $\hy_i$ close to major pullbacks $M$ and
$M'$ of $m$, and the short pullback $\hc_i$ of $\hy_i$. Since $\hy_i$ is the
shortest leaf in its orbit, it cannot map under itself or under $-\hy_i$. By
Lemma \ref{simple-pull-back-lam}, the leaf $\hc_i$ is a comajor, and we obtain a
sequence $\{\hc_i\}_{i=1}^{\infty}$ of preperiod 1 comajors
converging to $c$ such that $\hc_i\succ c$ for all $i$.
\end{proof}

\begin{corollary}\label{not-preperiodic}
Every not eventually periodic comajor $c$ is a two sided limit leaf in the
Cubic Symmetric Comajor Lamination $C_sCL$.
\end{corollary}

\begin{proof}
By Lemma \ref{cmajor-end-points}, the leaf $c$ is a two sided limit leaf in $\lam(c)$
approximated by leaves of $\lam(c)$ not sharing an endpoint with it. Thus, in
fact no leaf of $\lam(c)$ shares an endpoint with $c$. By Lemmas
\ref{appr-under} and \ref{appr-over}, we see that $c$ can be approximated on
both sides by a sequence of comajors in $C_sCL$ that do not share an endpoint
with $c$ as desired.
\end{proof}

Finally we consider \emph{preperiodic  comajors of preperiod bigger than 1}.

\begin{Lemma}\label{pre-period-bigger-than-1} A non-degenerate preperiodic
comajor $c$ of preperiod at least $2$ is a two sided limit leaf of $C_sCL$ or
an edge of a finite gap $H$ of $C_sCL$ whose edges are limits of comajors of
$C_sCL$ disjoint from $H$.
\end{Lemma}

\begin{proof}
Critical sets of the symmetric lamination $\lam(c)$ are collapsing
quadrilaterals $Q$ and $-Q$. We claim that all gaps of $\lam(c)$ are finite.
Indeed, let $U$ be an infinite gap of $\lam(c)$. By Theorem \ref{t:kiwi}, we
may assume that $U$ is periodic.
If 
the degree of $U$ is 
greater than $1$, then $\si_3^n(U)$ contains $Q$ or $-Q$ for some $n>0$, a
contradiction. If 
the degree of $U$ is 
$1$, then, by Theorem \ref{t:gaps-1}, the gap $\si_3^n(U)$ has a critical edge,
 again a contradiction.
Thus, all gaps of $\lam(c)$ are finite.

Since the minors $\pm m$ of $\lam(c)$ are not periodic (the preperiod of $c$ is
greater than $1$), if $\pm m$ are
two-sided limit leaves of $\lam(c)$, 
 then, by Lemmas \ref{appr-under} and \ref{appr-over},
 the leaves $\pm c$ are 
 two-sided limit leaves of $C_sCL$.
Assume now that
 $m$ is an edge of a finite gap $G$ of $\lam(c)$; let $G(c)$ be its pullback
containing $c$ and $G(M)$ be its pullback containing the majors. Then $\si_3$
maps $G(c)$ onto $G$ one-to-one, and sets $G$, $G(M)$, $G(c)$ are
non-periodic; $G(M)\supsetneq Q$ (hence, $G(M)$ is not a gap of $\lam(c)$) and
contains no diagonals that are leaves of $\lam(c)$.

We claim that each edge of $G(c)$ and $-G(c)$ is a comajor of a symmetric
lamination. Remove from $\lam(c)$ the edges of $Q$ that are not edges of
$G(M)$ and all its pullbacks, do the same with $-Q$, and thus construct a
lamination $\lam'(c)$ with critical sets $G(M)$ and $-G(M)$.

Let $\ell$ be an edge of $G(c)$. The sibling leaves of $\ell$ are edges of
$G(M)$; form a quadrilateral $Q'\subset G(M)$ by connecting their endpoint
(i.e., subdivide $G(M)$ by adding $Q'\subset G(M)$). Do the same with
$-G(M)$. By adding all preimages of the new leaves inside preimages of
$G(M)$, we obtain a new symmetric lamination with $\ell$ and $-\ell$ as
comajors.

The edges of $G(c)$ (or $-G(c)$) form a gap of $C_sCL$ since by Corollary
\ref{no-diagonal}, no diagonal of the polygon $G(c)$ can be a leaf (let alone
comajor!) of a symmetric lamination. We claim that all edges of $G(c)$ (and
$-G(c)$) are non-isolated in $\lam(c)$. Indeed, otherwise there exists a
finite gap $H$ that shares an edge (leaf) $\ell$ with $G(c)$. Consider cases.

(1) The gap $H$ is not an iterated pullback of $\pm Q$. By Theorem
\ref{nwt-thm}, the gap $H$ is preperiodic. Combining $H$ and $G(c)$ we obtain
in the end a periodic polygon subdivided into several smaller polygons.
Removing leaves located inside it, and all their iterated pullbacks, we will
obtain a symmetric lamination with some periodic gap so that a  diagonal can
be added to the lamination, a contradiction with Corollary \ref{no-diagonal}.

(2) The gap $H$ is an eventual pullback of $Q$ (or $-Q$).
Then $\ell$ maps
to an edge $\si_3^k(\ell)$ of $Q$ and the set $\si_3^k(G(c))$ is attached to
this edge. Since no image of $c$ is contained in $\sh(M)$, either $M$ or $M'$ must be
an edge of $\si_3^k(H)$.
If $\si_3^k(H)$ is contained in $G(M)$, then it is periodic,
and hence $M$ is periodic, a contradiction. If $\si_3^k(H)$ is not contained
in $G(M)$, then the image of $\si_3^k(H)\cup Q\cup G(M)$ is a preperiodic
polygon which contains $\si_3(M)$ as a diagonal, As in (1), this yields a
contradiction.

By (1) and (2), all edges of $G(c)$ (and $-G(c)$) are non-isolated in
$\lam(c)$. By Lemmas \ref{appr-under} and \ref{appr-over}, all the edges of
$G(c)$ (and $-G(c)$) are approximated by a sequence of leaves in $C_sCL$,
too.

Finally, we claim that none of these approximating comajors share an endpoint
with edges of $G(c)$ (and $-G(c)$). If they did, they would all have the same
preperiod and same period by Lemma \ref{cmajor-end-points}. Any two such
leaves create a fixed return triangle contradicting Proposition \ref{P-or-P}.
\end{proof}

\begin{theorem}[Main theorem of this section]\label{t:main} The Symmetric Cubic Comajor Lamination $C_sCL$ is a q-lamination.
\end{theorem}

\begin{proof} By Lemma \ref{pre-period-1}, Corollary \ref{not-preperiodic} and
Lemma \ref{pre-period-bigger-than-1}, no more than two comajors meet at a
single point. Hence, $C_sCL$ is a q-lamination.
\end{proof}

\noindent\textbf{Acknowledgments.} The results of this paper were presented
by the authors at the Lamination Seminar at UAB. It is a pleasure to express
our gratitude to the members of the seminar for their useful remarks and
suggestions.

\renewcommand\bibname{LIST OF REFERENCES}

\begin{thebibliography}{A}

\bibitem[BMOV13]{bmov13} A. Blokh, D. Mimbs, L. Oversteegen, K. Valkenburg,
    \emph{Laminations in the language of leaves}, Trans. Amer. Math. Soc.,
    \textbf{365} (2013), 5367--5391.


\bibitem[BOPT17]{bopt14} A. Blokh, L. Oversteegen, R. Ptacek, V. Timorin,
    \emph{Combinatorial Models for Spaces of Cubic Polynomials}, C. R. Math.
    (Paris), \textbf{355} (2017), 590--595.

\bibitem[BOPT19]{bopt19}A. Blokh, L. Oversteegen, R. Ptacek, V. Timorin,
    \emph{Models for spaces of dendritic polynomials}, Trans. Amer. Math. Soc.
    \textbf{372} (2019), 4829--4849

\bibitem[BOPT20]{bopt20} A. Blokh, L. Oversteegen, R. Ptacek, V. Timorin,
    \emph{Laminational models for some spaces of polynomials of arbitrary
    degree}, Memoirs of the AMS \textbf{265} (2020), No. 1288.

\bibitem[BOTSV2]{botsv2} A. Blokh, L. Oversteegen, V. Timorin, N. Selinger,
    S. Vejandla, \emph{Lavaurs algorithm for cubic symmetric polynomials}, in
    preparation

\bibitem[BOTSV3]{botsv3} A. Blokh, L. Oversteegen, V. Timorin, N. Selinger,
    S. Vejandla, \emph{Cubic symmetric polynomials}, in
    preparation




\bibitem[Kiw02]{kiw02} J. Kiwi, \emph{Wandering orbit portraits},
    Trans. Amer. Math. Soc. \textbf{354} (2002), 1473--1485.

\bibitem[Lav89]{lav89} P. Lavaurs, \emph{Une description combinatoire de
    l'involution d\'efinie par M sur les rationnels \`a d\'enominateur impair},
    C. R. Acad. Sci. Paris S\'er. I Math.
    303 (1986), no. 4, 143–146.



\bibitem[Sch09]{sch09} D. Schleicher, \emph{Appendix: Laminations, Julia
    sets, and the Mandelbrot set}, in: ``Complex dynamics: Families and
    Friends'', ed. by D. Schleicher, A K Peters (2009), 111--130.

\bibitem[Thu19]{thu19} W.~Thurston, H. Baik, Y. Gao, J. Hubbard, T. Lei, K.
    Lindsey, D. Thurston, \emph{Degree $d$ invariant laminations}, Ann. of
    Math. Stud. \textbf{205}, Princeton Univ. Press, Princeton, NJ, [2020],
    259--325.

\bibitem[Thu85]{thu85} W.~Thurston. \newblock {\em The
    combinatorics of iterated rational maps} (1985), with appendix by D. Schleicher,
    \emph{Laminations, Julia sets, and the Mandelbrot set}, published
    in: ``Complex dynamics: Families and Friends'', ed. by D.
    Schleicher, A K Peters (2009), 1--137.

\bibitem[Vej21]{Vej21} S. Vejandla, \emph{Cubic symmetric laminations}, PhD
    Thesis (2021), University of Birmingham at Alabama.
\end{thebibliography}

\end{document}